\documentclass[reqno,a4paper]{amsart}

\usepackage[latin1]{inputenc}
\usepackage[italian,english]{babel}
\usepackage{eucal,amsfonts,amssymb,amsmath,amsthm,epsfig,mathrsfs}
\usepackage{dsfont,cancel,soul}
\usepackage{color}
\usepackage{amscd,amsxtra}
\usepackage{enumerate}
\usepackage{latexsym}
\usepackage{bm}

\allowdisplaybreaks

\newcounter{ipotesi}

\Alph{ipotesi}
 \makeatletter \@addtoreset{equation}{section}

\makeatother \makeatletter

\newtheorem{thm}{Theorem}[section]
\newtheorem{hyp}[thm]{Hypotheses}{\rm}
\newtheorem{hyp0}[thm]{Hypothesis}{\rm}
\newtheorem{lemm}[thm]{Lemma}
\newtheorem{coro}[thm]{Corollary}
\newtheorem{prop}[thm]{Proposition}
\newtheorem{defi}[thm]{Definition}
\newtheorem{rmk}[thm]{Remark}{\rm}
\newtheorem{example}[thm]{Example}

\newcounter{parentenv}

\newcommand{\R}{{\mathbb R}}

\newcommand{\N}{{\mathbb N}}

\newcommand{\Rd}{\mathbb R^d}

\newcommand{\Rm}{\mathbb R^m}

\newcommand{\T}{{\bf T}}

\newcommand{\g}{{\bf g}}
\newcommand{\f}{{\bf f}}
\newcommand{\uu}{{\bf u}}
\newcommand{\A}{\bm{\mathcal A}}

\newcommand{\vv}{{\bf v}}
\newcommand{\ww}{{\bf w}}
\newcommand{\one}{\mbox{$1\!\!\!\;\mathrm{l}$}}

\newcommand{\h}{{\bf h}}

\begin{document}

\title[On invariant measures associated to weakly coupled systems]{On invariant measures associated to weakly coupled systems of Kolmogorov equations}
\thanks{The authors are members of G.N.A.M.P.A. of the Italian Istituto Nazionale di Alta Matematica (INdAM). Work partially supported by the INdAM-GNAMPA Project 2017 ``Equazioni e sistemi di equazioni di Kolmogorov in dimensione finita e non''.}
\author[D. Addona, L. Angiuli, L. Lorenzi]{Davide Addona, Luciana Angiuli, Luca Lorenzi}
\address{Dipartimento di Scienze Matematiche, Fisiche e Informatiche, Edificio di Matematica e Informatica, Universit\`a di Parma, Parco Area delle Scienze 53/A, I-43124 PARMA (Italy)}
\address{Dipartimento di Matematica e Fisica ``Ennio De Giorgi'', Universit\`a del Salento, Via per Arnesano, I-73100 LECCE (Italy)}
\address{Dipartimento di Scienze Matematiche, Fisiche e Informatiche, Edificio di Matematica e Informatica, Universit\`a di Parma, Parco Area delle Scienze 53/A, I-43124 PARMA (Italy)}
\email{d.addona@campus.unimib.it}
\email{luciana.angiuli@unisalento.it}
\email{luca.lorenzi@unipr.it}

\date{}

\keywords{Systems of elliptic operators with unbounded coefficients, invariant measures, semigroups of bounded operators, estimates for the spatial derivatives, asymptotic behaviour}
\subjclass[2000]{Primary: 35K40; Secondary 35K45, 35B40}

\begin{abstract}
In this paper, we deal with weakly coupled elliptic systems $\bm\A$ with unbounded coefficients. We prove the existence and characterize all the
systems of invariant measures for the semigroup $(\T(t))_{t\ge 0}$ associated to $\bm\A$ in $C_b(\Rd;\R^m)$. We also show some relevant properties of
the extension of $(\T(t))_{t\ge 0}$ to the $L^p$-spaces related to systems of invariant measures. Finally, we study the asymptotic behaviour of $(\T(t))_{t\ge 0}$
as $t$ tends to $+\infty$.
\end{abstract}

\maketitle

\section{Introduction}

In the last two decades, partial differential equations with unbounded coefficients
have attracted the attention of many researchers, for their remarkable applications
in economy and finance and for their strong connection with the theory of
stochastic differential equations.
Such equations appear also in the analysis of the weighted $\overline{\partial}$-problem in $\mathbb C^d$, in the time-dependent Born-Openheimer theory
and also in the study of Navier-Stokes equations. (We refer the interested reader to \cite{AALT, BGT, Dall, Ha-Rh,  Ha-He, HRS, hieber} for further details.)
In particular, the
Cauchy problems associated to second-order differential equations of elliptic and parabolic type
have been widely studied in the classical setting of bounded and continuous functions and in $L^p$-spaces, related to the Lebesgue measure
and to the so-called {\it invariant measures}.
The literature is nowadays rather rich in the case of a single equation (we refer the interested reader to \cite{newbook} for further details).
On the other hand, according to our knowledge less is known about the theory of systems (we refer the interested reader to \cite{AALT, AngLorPal, DelLor11OnA, hetal09}) and, in particular, invariant measures for systems seem to have not been studied so far.

In this paper, we consider weakly coupled elliptic operators $\bm {\mathcal A}$ defined on
smooth functions $\boldsymbol\zeta:\Rd\to\Rm$, ($m\ge 2$), by
\begin{align}\label{picco}
(\A\boldsymbol\zeta)(x)& =\sum_{i,j=1}^dq_{ij}(x)D_{ij}\boldsymbol\zeta(x)+
\sum_{j=1}^db_j(x)D_j\boldsymbol\zeta(x)+C(x)\boldsymbol\zeta(x)\notag\\
& = {\rm Tr}(Q(x)D^2\boldsymbol\zeta(x))+ \langle {\bf b}(x), \nabla \boldsymbol\zeta(x)\rangle+ C(x)\boldsymbol\zeta(x).
\end{align}

The results in \cite{AALT,AngLorPal,DelLor11OnA} show, that under mild assumptions on the coefficients
$q_{ij},b_i:\Rd\to\Rm$ ($i,j=1,\ldots,d$) and $C:\Rd\to\R^{2m}$, and assuming the existence of a so-called Lyapunov function $\varphi$
for the scalar operator $\mathcal A=\sum_{i,j=1}^d q_{ij}D_{ij}+\sum_{i=1}^db_iD_i$ (see Hypothesis \ref{hyp-base}(iv)), it is possible to associate
a semigroup $(\T(t))_{t\ge 0}$ to $\bm{\mathcal A}$ in $C_b(\Rd;\Rm)$, the space of bounded and continuous functions $\f:\Rd\to\R^m$. The semigroup
is defined in the natural way: for any $\f\in C_b(\Rd;\R^m)$ and $t>0$ $\T(t)\f$ is the value at $t$ of the unique bounded classical solution of the
Cauchy problem
\begin{equation}
\left\{
\begin{array}{ll}
D_t\uu=\A\uu, &{\rm in}~(0,+\infty)\times\R^d,\\[1mm]
\uu(0,\cdot)=\f, &{\rm in}~\Rd.
\end{array}
\right.
\label{cantare}
\end{equation}

A variant of the classical maximum principle, based on the existence of the function $\varphi$ can be used to show that, for any $t>0$, $x\in \Rd$
and $p\in (1,+\infty)$
\begin{equation}\label{usiamo_label_intro}
|(\T(t)\f)(x)|^p\le (T(t)|\f|^p)(x),\qquad\;\, \f\in C_b(\Rd;\Rm),
\end{equation}
(see \cite[Proposition 2.8]{AngLorPal}).

%Formula \eqref{usiamo_label_intro} and the contractivity of $(T(t))_{t\ge 0}$ in $C_b(\Rd)$ allow to overcome
%one of the first difficulties concerning the Cauchy problems associated to parabolic systems,
%namely the uniqueness of a classical bounded solution,
%or equivalently the validity of some maximum principles which would allow to estimate the sup-norm of the solution
%in terms of the sup-norm of the initial datum.
%Results of this type are known in the case of bounded coefficients only for some class of systems.
%In \cite{KreMaz12Max} the authors present results pertaining to various versions of the maximum principle
%for elliptic and parabolic systems of arbitrary order. In particular, necessary and
%sufficient conditions for validity of the classical maximum modulus principles for systems of second-order are given.
%In \cite{protter} a componentwise maximum principle is proved for weakly coupled parabolic system when the
%entries of the matrix $C$, more precisely the off-diagonal terms and the sum of each row, have to satisfy some sign conditions.
%
%A variant to the approach considered in \cite{KreMaz12Max} has been proposed in \cite{AALT} where
%an $L^\infty$-estimate has been proved for
%the evolution operators associated to uniformly elliptic operators with unbounded and time-dependent
%coefficients, in a greater generality than that considered here.

Differently from the case of bounded and continuous coefficients, the analysis of Markov semigroups on $L^p$-spaces is much more difficult.
In \cite{AngLorPal} a class of nonautonomous parabolic first-order coupled systems has been considered in
the Lebesgue space $L^p(\Rd; \Rm)$, $p\in [1,+\infty)$.
Sufficient conditions, consisting of quite strong growth assumptions on the coefficients of the elliptic operator $\A$,
have been supplied to guarantee that the associated evolution operator extends to $L^p(\Rd;\Rm)$. Such growth
assumptions are not merely technical conditions. Indeed, already in the scalar case, the Cauchy problem \eqref{cantare} may be not well posed in the usual $L^p$-spaces if the coefficients of the elliptic operator $\mathcal A$ are unbounded, unless they satisfy rather restrictive growth assumptions. The scalar case also shows that
a way to deal with $L^p$-spaces, under reasonable assumptions on the coefficients of the elliptic operator, is to replace the Lebesgue measure by
another measure, possibly absolutely continuous with respect to the Lebesgue one. The best situation in the scalar case
is when an invariant measure $\mu $ exists, which is a Borel probability measure such that
\begin{eqnarray*}
\int_{\R^d}T(t)f\,d\mu = \int_{\R^d} f\,d\mu , \qquad\;\, t>0, \;\,f\in C_b(\R^d),
\end{eqnarray*}
where $(T(t))_{t\ge 0}$ is the Markov semigroup naturally associated to the elliptic operator $\mathcal A$ in $C_b(\Rd)$. Under quite mild assumptions,
a unique invariant measure exists, it is equivalent to the Lebesgue measure
and is related to the asymptotic behaviour of $T(t)$, since
\begin{eqnarray*}
\lim_{t\to +\infty} (T(t)f)(x) = \int_{\R^d} f\,d\mu, \qquad\;\, f\in C_b(\R^d), \;\, x\in \R^d.
\end{eqnarray*}
Moreover, the operators $T(t)$ may easily  be extended to contractions in $L^p_{\mu}(\R^d)$, the $L^p$-space associated with the measure $\mu$, for every
$p\in [1, +\infty)$.

In this paper we give a consistent definition of invariant measures for the semigroup $(\T(t))_{t\ge 0}$ in $C_b(\Rd;\Rm)$,
providing sufficient conditions for the existence of such measures and proving that, as in the scalar case,
 the vector valued semigroup $(\T(t))_{t\ge 0}$ enjoys good properties in the $L^p$-spaces related to these measures.
It seems quite natural to expect that the single measure $\mu$ associated to a single equation
in the scalar case is replaced by an $m$-dimensional vector of measures associated to the
$m$ equations of the system.
We call \emph{system of invariant measures} for the semigroup $(\T(t))_{t\ge 0}$, a
family of {\it positive} and finite Borel measures $\{\mu_i: i=1,\ldots,m\}$ over $\Rd$ satisfying
\begin{align*}
\sum_{i=1}^m\int_{\Rd}(\T(t)\f)_id\mu_i=\sum_{i=1}^m\int_{\Rd}f_id\mu_i,\qquad\;\,\f\in C_b(\Rd;\Rm).
\end{align*}
We assume that the off-diagonal
entries of the matrix valued function $C$ are nonnegative functions (see Hypothesis \ref{hyp-base}(v)).
This additional assumption, in particular, implies that the semigroup $(\T(t))_{t\ge 0}$ is nonnegative in the sense that,
if the entries of the function $\f$ are all nonnegative, then $\T(t)\f$ has nonnegative components as well, for any $t>0$.
The componentwise positiveness of the semigroup $(\T(t))_{t\ge 0}$ is essential in our analysis to prove the existence
of a system of invariant measures. This is the reason why we confine ourselves to weakly coupled elliptic operators $\A$.

About existence and uniqueness of systems of invariant measures, Theorem \ref{exi_inv_meas} shows that, under reasonable assumptions,
there exists a unique (up to a multiplicative constant) system of invariant measures for the semigroup $(\T(t))_{t\ge 0}$. More precisely,
 if $\{\mu_j: j=1,\ldots,m\}$ is a system of invariant measures for $(\T(t))_{t\ge 0}$, then there exists
 a positive constant $c$ such that $\mu_j=c\xi_j\mu$ for any $j=1,\ldots,m$, where $\mu$ is the
 invariant measure associated to the scalar semigroup $(T(t))_{t\ge 0}$ and  $\xi=(\xi_1,\ldots,\xi_m)$ is a not trivial constant vector
which belongs to $\bigcap_{x\in \Rd}{\rm Ker}(C(x))$.
This crucial assumption together with the non positivity of the quadratic form associated to $C$ (see Hypothesis \ref{hyp-base}(iii))
 are inspired by the scalar case where the existence (and consequently the uniqueness)
of an invariant measure is guaranteed when the potential term of the elliptic operator identically vanishes on $\Rd$.
See Remark \ref{rmk-sign} for further details.

Formula \eqref{usiamo_label_intro} yields immediately that $(\T(t))_{t\ge 0}$ extends to a strongly continuous semigroup in $L^p_{\bm\mu}(\Rd;\Rm)=\bigotimes_{i=1}^mL^p_{\mu_i}(\Rd)$ for any $p\in[1,+\infty)$.
Under additional growth assumptions on the coefficients of $\A$, we prove some pointwise estimate for the first- and second-order spatial derivatives of $\T(t)\f$. More precisely, we show that
\begin{align}
\label{der_est-usiamo questa nella sezione}
|D^k_x\T(t)\f|^p \le \Gamma_{p,k,h}(t)T(t)\bigg(\sum_{j=0}^h|D^j\f|^2\bigg)^{\frac{p}{2}}
\end{align}
in $\Rd$ for any $t>0$, $\f\in C^h_b(\Rd,\R^m)$, $p>1$, $k\in\{1,2\}$ and $h\in\{0,\ldots,k\}$, where $\Gamma_{p,k,h}$ is a positive function defined in $(0,+\infty)$,
whose behaviour as $t$ tends to $0^+$ is sharp.
Clearly, in this case each $\T(t)$ is a bounded operator from $L^p_{\bm\mu}(\Rd;\R^m)$ into $W^{2,p}_{\bm\mu}(\Rd;\R^m)$
(the set of all functions $\f\in L^p_{\bm\mu}(\Rd;\Rm)$ whose distributional derivatives up to the second-order are in $L^p_{\bm\mu}(\Rd;\Rm)$)
for any $p\in (1,+\infty)$. Estimate \eqref{der_est-usiamo questa nella sezione} with $k=1$ is useful also to provide a partial characterization of the domain $D({\bf A}_p)$ of the infinitesimal generator ${\bf A}_p$ of the semigroup $(\T(t))_{t\ge 0}$ in $L^p_{\bm\mu}(\Rd;\Rm)$, since it allows us to prove that
$D({\bf A}_p)\subset W^{1,p}_{\bm\mu}(\Rd;\Rm)$. The complete characterization of $D({\bf A}_p)$ is out of the scope of this paper and, as the scalar case shows, it is
known only in some particular cases.

Finally, we relate the system of invariant measures to the asymptotic behaviour of the function $\T(t)\f$ as $t \to +\infty$.
More precisely, we assume that $|q_{ij}(x)|\le c|x|^2\varphi(x)$ ($i,j=1, \ldots,d$) and $\langle {\bf b}(x),x\rangle\le c|x|^2\varphi(x)$ as $|x|\to +\infty$, for
 some positive constant $c$, and we show that
$\T(t)\f$ converges to ${\mathcal M}_{\f}\xi:=\left (\sum_{j=1}^m\int_{\Rd}f_jd\mu_j\right )\xi$, locally uniformly in $\Rd$ as $t\to +\infty$.
As a byproduct, we deduce that, if  $\f\in L^p_{\bm\mu}(\Rd;\Rm)$, then the function $\T(t)\f$ converges to ${\mathcal M}_{\f}\xi$
also in $L^p_{\bm\mu}(\Rd;\R^m)$ as $t\to +\infty$.

The plan of the paper is the following. First in Section \ref{sect:2} we introduce some known results on equations and systems of elliptic operators and prove some basic facts which are crucial in all our analysis. Section \ref{sect:3} is devoted to the systems of invariant measures, the analysis of the semigroup $(\T(t))_{t\ge 0}$ in
$L^p$-spaces associated with systems of invariant measures and pointwise estimates for the spatial derivatives up to the second-order of the function $\T(t)\f$. Finally, in Section \ref{sect:4} we study the long time behaviour of the function $\T(t)\f$ when $\f$ is
bounded and Borel measurable and when it belongs to $L^p_{\bm\mu}(\Rd;\Rm)$ for some $p\in [1,+\infty)$.

\subsection*{Notation}
Functions with values in $\R^m$ are displayed in bold style. Given a function $\f$ (resp. a sequence
$(\f_n)$) as above, we denote by $f_i$ (resp. $f_{n,i}$) its $i$-th component (resp. the $i$-th component
of the function $\f_n$).
By $B_b(\Rd;\Rm)$ we denote the set of all the bounded Borel measurable functions $\f:\Rd\to\Rm$, where $\|f\|_{\infty}^2=\sum_{k=1}^m\sup_{x\in\Rd}|f_k(x)|^2$.
For any $k\ge 0$, $C^k_b(\R^d;\R^m)$ is the space of all the functions
whose components belong to $C^k_b(\R^d)$, where ``$b$'' stays for bounded. Similarly, we use the subscript ``$c$'' and ''$0$'' for spaces of functions with compact support and spaces of functions vanishing at infinity, respectively.
When $k\in (0,1)$ we use the subscript ``loc'' to denote the space of all $f\in C(\Rd)$
which are H\"older continuous in any compact set of $\Rd$.
We assume that the reader is familiar with the parabolic spaces $C^{h+\alpha/2,k+\alpha}(I\times \Rd)$
($\alpha\in [0,1)$, $h,k\in\N\cup\{0\}$), and we use the subscript ``loc'' with the same meaning as above.
The symbols $D_tf$, $D_i f$ and $D_{ij}f$, respectively, denote the time derivative $\frac{\partial f}{\partial t}$ and the spatial derivatives $\frac{\partial f}{\partial x_i}$ and $\frac{\partial^2f}{\partial x_i\partial x_j}$ for any $i,j=1, \ldots,d$.
For any $k\in\N$, we write $|D^k_x\uu|^2$ to denote the sum $\sum_{j=1}^m|D^k_xu_j|^2$. If $k=1$ we write $D_x\uu$ and $J_x\uu$ indifferently for the Jacobian matrix of $\uu$ with respect to the spatial variables.

By $e_j$ and $\one$ we denote, respectively, the $j$-th vector of the Euclidean basis of $\R^m$ and the function identically equal to $1$ in $\Rd$.
The open ball in $\Rd$ centered at $ 0$ with radius $r>0$ and its closure are denoted by  $B_r$ and $\overline B_r$, respectively.

\goodbreak

\section{Hypotheses and preliminary results}
\label{sect:2}
Throughout the paper, if not otherwise specified, we assume the following assumptions on the coefficients of the operator $\bm{\mathcal A}$ in \eqref{picco}.

\begin{hyp}
\label{hyp-base}
\begin{enumerate}[\rm (i)]
\item
The coefficients $q_{ij}=q_{ji}$, $b_j$ and the entries $c_{hk}$ of the nonidentically vanishing matrix
valued function $C$ belong to $C^{\alpha}_{\rm loc}(\Rd)$ for some $\alpha\in (0,1)$;
\item
the infimum $\mu_0$ over $\Rd$ of the minimum eigenvalue $\mu_Q(x)$ of the matrix $Q(x)=(q_{ij}(x))$ is positive;
\item
$\langle C(x)y,y\rangle\le 0$ for any $x\in\Rd$ and $y\in\Rm$;
\item
there exists a positive function $\varphi\in C^2(\Rd)$, blowing up as $|x|\to+\infty$
such that ${\mathcal A}\varphi(x)\le a-c\varphi(x)$ for any $x\in\Rd$ and some positive constants $a,c$,
where ${\mathcal A}={\rm Tr}(QD^2)+\langle {\bf b},\nabla\rangle$;
\item
the off-diagonal entries of the matrix valued function $C$ are nonnegative;
\item
there exists $0\neq\xi\in\Rm$ such that $\xi\in {\rm Ker}(C(x))$ for any
$x\in\Rd$;
\item
there does not exist a nontrivial set $K\subset \{1, \ldots, m\}$ such that the coefficients $c_{ij}$ identically vanish on $\Rd$ for any $i\in K$ and $j\notin K$.
\end{enumerate}
\end{hyp}

In the following Lemma \ref{carusorompi}, Theorem \ref{thm-A2} and Proposition \ref{prop-rossana} we collect some basic consequences of
the previous assumptions.

\begin{lemm}\label{carusorompi}
Let Hypotheses $\ref{hyp-base}(iii)$, $(vi)$ be satisfied. Then, the set equality
${\rm Ker}(C(x))={\rm Ker}((C(x))^*)$
holds true for any $x\in\Rd$. If, in addition, Hypothesis $\ref{hyp-base}(v)$ is satisfied, then
for any $x\in\Rd$ the spectrum of the matrix $C(x)$ is contained in the left-halfplane and $0$ is the unique
eigenvalue on the imaginary axis.
\end{lemm}

\begin{proof}
Fix $x\in\Rd$. Since $0$ is an eigenvalue of $C(x)$, it is also an eigenvalue of the adjoint matrix $(C(x))^*$.
Let $\xi_0$ be any vector such that $(C(x))^*\xi_0=0$ and, by contradiction, let us assume that $\eta:=C(x)\xi_0\neq 0$.
Let us fix $\beta>0$ and observe that
\begin{align*}
&\langle C(x)(\beta\xi_0+\eta),\beta\xi_0+\eta\rangle\\
=&
\beta^2\langle C(x)\xi_0,\xi_0\rangle+\beta \langle C(x)\xi_0,\eta\rangle
+\beta\langle C(x)\eta,\xi_0\rangle+\langle C(x)\eta,\eta\rangle\\
=&\beta^2\langle \xi_0,(C(x))^*\xi_0\rangle+\beta |\eta|^2
+\beta\langle \eta,(C(x))^*\xi_0\rangle+\langle C(x)\eta,\eta\rangle\\
=&\beta |\eta|^2+\langle C(x)\eta,\eta\rangle.
\end{align*}
It is clear that we can fix $\beta>0$ such that
$\langle C(x)(\beta\xi_0+\eta),\beta\xi_0+\eta\rangle>0$ getting to a contradiction.
The inclusion ${\rm Ker}((C(x))^*)\subset {\rm Ker}(C(x))$ follows.
Since $\langle C(x)y,y\rangle=\langle (C(x))^*y,y\rangle$ for any $y\in\R^m$, the same arguments above applied to $(C(x))^*$
yield the other inclusion ${\rm Ker}(C(x))\subset {\rm Ker}((C(x))^*)$.

Let us complete the proof by checking the last statement.
To begin with, we observe that $C(x)$ has not eigenvalues $\lambda$ with positive real part. This is clear if $\lambda$ is real. Indeed, denoting by $\eta$
a corresponding unit eigenvector, we would get $0<\lambda=\langle C(x)\eta,\eta\rangle$ contradicting Hypothesis \ref{hyp-base}(iii).
If $\lambda\in \mathbb C\setminus\R$ and $\eta$ is as above, then $\eta=\eta_1+i\eta_2$ for some $\eta_1,\eta_2\in\R$. It is immediate to check that
$0<{\rm Re}\lambda=\langle C(x)\eta_1,\eta_1\rangle+\langle C(x)\eta_2,\eta_2\rangle$, again contradicting Hypothesis \ref{hyp-base}(iii).

To prove that $0$ is the unique eigenvalue of $C(x)$ on the imaginary axis, we fix $\lambda$ sufficiently large such that
$\lambda+\mu>0$ for any real eigenvalue of $C(x)$ and $\lambda+c_{ii}(x)>0$ for any $i=1,\ldots,m$. With this choice
all the elements of the matrix $C_{\lambda}(x):=C(x)+\lambda I$ are nonnegative. Moreover, since $\sigma(C_{\lambda}(x))=\sigma(C(x))+\lambda$,
all the real eigenvalues of $C_{\lambda}(x)$ are positive and $\lambda$ is the greatest one.
A generalization of Perron-Frobenius Theorem (see \cite[Theorem 2.7]{gersh})
implies that the spectral radius of $C_{\lambda}(x)$ (i.e., the maximum of the moduli of the eigenvalues of $C_{\lambda}(x)$) belongs to $\sigma(C_{\lambda}(x))$. From this
and the above remarks it follows that $\lambda$ is the maximum of the eigenvalues of $C_{\lambda}(x)$. Coming back to $C(x)$, we conclude that this matrix has not
nontrivial eigenvalues on the imaginary axis and we are done.
\end{proof}

\begin{rmk}\label{nosym}{\rm
We stress that our assumptions on $C$ in general do not imply that $C(x)$ is symmetric for some $x\in \Rd$.
Indeed, it is immediate to check that the matrix valued function $C$ defined by
%\begin{align*}
%C=
%\begin{pmatrix}
%-1 & 1/2 & 1/2\\
%1/3 & -1 & 2/3 \\
%2/3 & 1/2 & -7/6
%\end{pmatrix},
%\end{align*}
\begin{align*}
C(x)=
\begin{pmatrix}
-\zeta_1(x)-\zeta_2(x) & \zeta_1(x) & \zeta_2(x)\\
\zeta_2(x) & -\zeta_1(x)-\zeta_2(x)-\zeta_3(x) & \zeta_1(x)+\zeta_3(x) \\
\zeta_1(x) & \zeta_2(x)+\zeta_3(x) & -\zeta_1(x)-\zeta_2(x)-\zeta_3(x)
\end{pmatrix}
\end{align*}
for any $x\in\Rd$, satisfies Hypotheses \ref{hyp-base}(iii), (v)-(vii) for any triplet of positive locally H\"older continuous functions $\zeta_1,\zeta_2,\zeta_3:\Rd\to\R$.}
\end{rmk}

Under Hypotheses \ref{hyp-base}(i)-(iv), for any $\f\in C_b(\Rd;\R^m)$ the Cauchy problem
\begin{equation}
\left\{
\begin{array}{lll}
D_t\uu(t,x)=\A\uu(t,x), & t\in(0,+\infty), &x\in\Rd,\\[1mm]
\uu(0,x)=\f(x), && x \in \Rd
\end{array}
\right.
\label{pioggia}
\end{equation}
admits a unique classical solution
$\uu \in C^{1,2}((0,+\infty)\times\Rd;\Rm)\cap C_b([0,+\infty)\times\Rd;\Rm)$. Function $\uu$
satisfies the estimate $\|\uu\|_{\infty}\le\|\f\|_{\infty}$ and can be obtained equivalently as the limit in $C^{1,2}_{\rm loc}((0,+\infty)\times\Rd;\Rm)$
\begin{enumerate}[\rm (i)]
\item
of the sequence $(\uu_n)$ of classical solutions to the Cauchy-Dirichlet problem
\begin{equation}\label{prob_approx_Dir}
\left\{
\begin{array}{lll}
D_t\uu_n(t,x)=\boldsymbol{\mathcal A}\uu_n(t,x), \qquad\quad& t\in (0,+\infty), \;\,x\in B_n,\\[1mm]
\uu_n(t,x)={\bf 0}, & t\in (0,+\infty),\;\,x\in\partial B_n,\\[1mm]
\uu_n(0,x)=\f(x), & x\in B_n;
\end{array}
\right.
\end{equation}
\item
of the sequence $({\bf v}_n)$ of classical solutions to the Cauchy-Neumann problem
\begin{equation}\label{prob_approx_Neu}
\left\{
\begin{array}{lll}
D_t\vv_n(t,x)=\boldsymbol{\mathcal A}\vv_n(t,x), \qquad\quad& t\in (0,+\infty), \;\,x\in B_n,\\[1mm]
\displaystyle\frac{\partial\uu_n}{\partial\nu}(t,x)={ \bf 0}, & t\in (0,+\infty),\;\,x\in\partial B_n,\\[2mm]
\vv_n(0,x)=\f(x), & x\in\overline B_n,
\end{array}
\right.
\end{equation}
where $\nu$ denotes the unit exterior normal vector to $\partial B_n$.
\end{enumerate}
We refer the reader to \cite{AALT,AngLorPal,DelLor11OnA} for more details.

The above result allowed the authors
of \cite{DelLor11OnA} to associate a semigroup $(\T(t))_{t\ge 0}$ (in the sequel simply denoted by $\T(t)$) of bounded operators in $C_b(\Rd;\Rm)$ with the operator $\A$ in \eqref{picco}: for any $\f\in C_b(\Rd;\Rm)$
and $t>0$, $\T(t)\f$ is the value at $t$ of the unique bounded classical solution to problem \eqref{pioggia}.
In \cite[Theorem 3.2]{AALT}, actually in a greater generality, it has been proved that the semigroup $\T(t)$
admits an integral representation formula in terms of some finite Borel measures. More precisely,
\begin{equation}
(\T(t)\f)_i(x)=\sum_{j=1}^m\int_{\Rd}f_j(y)p_{ij}(t,x,dy),\qquad\;\,\f\in C_b(\Rd;\R^m).
\label{for_rep_1}
\end{equation}
The measures $p_{ij}(t,x,dy)$ are absolutely continuous with respect to the Lebesgue measure but, in general, differently from the scalar case,
they are signed measures.
Through formula \eqref{for_rep_1} the semigroup $\T(t)$ can be extended to
$B_b(\Rd;\R^m)$, with a strong Feller semigroup (i.e., $\T(t)\f$ belongs to $C_b(\Rd)$ for any $t>0$ and $\f\in B_b(\Rd;\R^m)$; actually, $\T(t)\f$ belongs to $C^2(\Rd)$
as a consequence of interior Schauder estimates).
Moreover,
\begin{equation}
|\T(t)\f|^p\le T(t)|\f|^p,\qquad\;\,t>0,\;p>1,\;\,\f\in C_b(\Rd;\Rm),
\label{vector-scalar}
\end{equation}
where $T(t)$ is the semigroup of contractions associated in $C_b(\Rd)$ to the operator ${\mathcal A}$ (see Hypothesis \ref{hyp-base}(iv)).
More precisely (see \cite[Chapters 1 \& 9]{newbook}),
\begin{thm}
\label{thm-valle}
Under Hypothesis $\ref{hyp-base}(i)$, $(ii)$ and $(iv)$ there exists a Markov contraction semigroup $T(t)$ associated to ${\mathcal A}={\rm Tr}(QD^2)+\langle {\bf b},\nabla\rangle$ in $C_b(\Rd)$. For any $f\in C_b(\Rd)$, $T(\cdot)f$ is the unique solution in $C_b([0,+\infty)\times\Rd)\cap C^{1,2}((0,+\infty)\times\Rd)$
to the differential equation $D_tu-{\mathcal A}u=0$, which satisfies the condition $u(0,\cdot)=f$.
For any $f\in C_b(\Rd)$ it holds that
\begin{eqnarray}\label{rep_int}
(T(t)f)(x)=\int_{\Rd}f(y)p(t,x,dy),\qquad\;\,t>0,\;\,x\in\Rd,
\end{eqnarray}
where each $p(t,x,dy)$ is a Borel probability measure which admits a strictly positive density with respect to the Lebesgue measure.
As a byproduct, if $f\ge 0$ does not identically vanishes in $\Rd$, then
$T(\cdot)f$ is strictly positive in $(0,+\infty)\times\Rd$ and
\begin{eqnarray*}
|T(t)f|^p\le T(t)|f|^p,
\end{eqnarray*}
in $\Rd$ for any $t>0$, $p\in (1,+\infty)$ and $f\in C_b(\Rd)$.
Finally, there exists a unique invariant measure $\mu$ associated with the semigroup $T(t)$, i.e., there exists a unique Borel probability measure $\mu$ such that
\begin{eqnarray*}
\int_{\Rd}T(t)fd\mu=\int_{\Rd}fd\mu,\qquad\;\,f\in C_b(\Rd),\;\,t>0.
\end{eqnarray*}
\end{thm}

\begin{rmk}
\label{rmk-2.5}
{\rm Some of the results in Theorem \ref{thm-valle} can be extended also to the case of elliptic operators with a nontrivial potential term.
More precisely, if ${\mathcal A}_c={\rm Tr}(QD^2)+\langle {\bf b},\nabla\rangle+c$, where the diffusion and drift coefficients
$q_{ij}$ and $b_j$ $(j=1,\ldots,d)$ satisfy Hypotheses \ref{hyp-base} and $c\in C^{\alpha}_{\rm loc}(\Rd)$ is bounded from above, then
`the Cauchy problem
\begin{equation}
\left\{
\begin{array}{lll}
D_tu(t,x)={\mathcal A}_cu(t,x), &t\in (0,+\infty), &x\in\Rd,\\[1mm]
u(0,x)=f(x), &&x\in\Rd,
\end{array}
\right.
\label{terzo}
\end{equation}
admits (at least) one solution $u\in C^{1+\alpha/2,2+\alpha}_{\rm loc}((0,+\infty)\times\Rd)\cap C([0,+\infty)\times\Rd)$ which satisfies
the estimate $\|u(t,\cdot)\|_{\infty}\le e^{c_0t}\|f\|_{\infty}$ for any $t>0$, where $c_0$ is the supremum over $\Rd$ of the function $c$.
Uniqueness may fail, but in any case, if $f\ge 0$, the above Cauchy problem admits a minimal solution $u$, in the sense that, if
$v$ is any other solution, then $v(t,x)\ge u(t,x)$ for any $(t,x)\in [0,+\infty)\times\Rd$. This allows to associate
a semigroup of bounded operators in $C_b(\Rd)$ with the operator ${\mathcal A}_c$: for any $f\in C_b(\Rd)$,
$T(\cdot)f=u_+-u_-$ where $u_+$ and $u_-$ are the minimum nonnegative solutions to the Cauchy problem
\eqref{terzo} with $f$ being replaced, respectively, by the positive and negative part of $f$.
Moreover, if $f\ge 0$ does not identically vanish on $\Rd$, then $T(\cdot)f$ is strictly positive on $(0,+\infty)\times\Rd$. See \cite{AngLor10Com} for further details.
}
\end{rmk}

In view of \eqref{vector-scalar} and Theorem \ref{thm-valle}, we conclude that $\T(t)$ is a contraction semigroup in $C_b(\Rd;\R^m)$, i.e.,
\begin{equation}
|(\T(t)\f)(x)|\le\|\f\|_{\infty},\qquad\;\,t>0,\;\,x\in\Rd.
\label{tiburtina}
\end{equation}

Hypothesis \ref{hyp-base}(v) is the key tool to prove the positivity of the semigroup $\T(t)$.
In the proof of the following Proposition \ref{posi} we shall make use of the following interior Schauder estimates.

\begin{thm}[Proposition A.1 \& Theorem A.2 of \cite{AALT}]
\label{thm-A2}
Let Hypotheses $\ref{hyp-base}(i)$, $(ii)$ hold. Let further $\uu\in C^{1+\alpha/2,2+\alpha}_{\rm loc}((0,T]\times\Rd;\R^m)$ satisfy the differential equation
$D_t\uu = \bm{\mathcal A}\uu+\g$ in $(0,T]\times\Rd$, for some $\g\in C^{\alpha/2,\alpha}_{\rm loc}((0,T]\times\Rd;\R^m)$ and $T>0$.
Then, for any $\tau\in (0,T)$ and any pair of bounded open sets $\Omega_1$ and $\Omega_2$ such that $\Omega_1$ is compactly supported in $\Omega_2$, there exists a positive
constant $K_1$, depending on $\Omega_1, \Omega_2, \tau, T$, but being independent of $\uu$, such that
\begin{align*}
&\|\uu\|_{C^{1+\alpha/2,2+\alpha}((\tau,T]\times\Omega_1;\R^m)}\notag\\
\le & K_1(\|\uu\|_{C_b((\tau/2,T]\times\Omega_2;\R^m)}+\|\g\|_{C^{\alpha/2,\alpha}((\tau/2,T]\times\Omega_2;\R^m)}).
\end{align*}
Further, if $\g\in C^{\alpha/2,\alpha}_{\rm loc}([0,T]\times\Rd;\R^m)$ and $\uu\in C^{1+\alpha/2,2+\alpha}_{\rm loc}((0,T]\times\Rd;\R^m)\cap C([0,T]\times\Rd)$,
then, for any $0<r_1<r_2$ there exists a positive constant $K_2$, depending on $r_1$, $r_2$ and $T$ but being independent of $\uu$, such that
\begin{align*}
&t\|D^2_x\uu(t,\cdot)\|_{C(\overline B_{r_1};\R^m)}+\sqrt{t}\|J_x\uu(t,\cdot)\|_{C(\overline B_{r_1};\R^m)}\\
\le &K_2(\|\uu\|_{C([0,T]\times\overline B_{r_2})}
+\|{\bf g}\|_{C^{\alpha/2,\alpha}([0,T]\times\overline B_{r_2};\R^m)})
\end{align*}
for any $t\in (0,T]$.
\end{thm}

Throughout the paper we shall make also use of the following local (in space) compactness property of the semigroup $\T(t)$ in $C_b(\Rd;\R^m)$.

\begin{prop}
\label{prop-rossana}
Under Hypotheses $\ref{hyp-base}(i)$-$(iv)$, for any bounded sequence $(\f_n)\subset C_b(\Rd;\R^m)$  and for any
$t_0>0$, there exists a subsequence $(\f_{n_k})$ such that $\T(\cdot)\f_{n_k}$ converges uniformly in $(t_0,+\infty)\times B_m$ for every $m>0$.
In particular, if $\f_n$ converges locally uniformly on $\Rd$ to $\f$, then
$\T(\cdot)\f_n$ converges uniformly in $(0,+\infty)\times B_r$ to $\T(\cdot)\f$ for any $r>0$.
\end{prop}

\begin{proof}
To begin with, we observe that the Schauder estimates in Theorem  \ref{thm-A2} show that the sequence
$(\T(t_0)\f_n)$ is bounded in $C^{2+\alpha}(B_r;\R^m)$ for every $r>0$. Hence, by Arzel\`a-Ascoli theorem and a compactness argument,
we can easily prove that there exists a subsequence $(\T(t_0)\f_{n_k})$ which converges locally uniformly in $\Rd$ to a function $\g\in C_b(\Rd;\R^m)$.

Next, we observe that the arguments in the proof of \cite[Lemma 5.3]{KunLorLun09Non} show that
$T(t)\varphi\le \varphi+c^{-1}a$ in $\Rd$ for every $t>0$, where $\varphi$ is the function in Hypothesis \ref{hyp-base}(iv). From this estimate we easily deduce that
$\sup_{t>0}p(t,x,\Rd\setminus B_r)$ tends to $0$, locally uniformly with respect to $x$, as $r$ tends to $+\infty$.
Indeed,
\begin{align*}
p(t,x,\Rd;\setminus B_r)=&\int_{\Rd\setminus B_r}p(t,x,dy)\le \frac{1}{\inf_{\Rd\setminus B_r}\varphi}\int_{\Rd}\varphi(y) p(t,x,dy)\\
\le &\frac{1}{\inf_{\Rd\setminus B_r}\varphi}(T(t)\varphi)(x)\le \frac{1}{\inf_{\Rd\setminus B_r}\varphi}(\varphi(x)+c^{-1}a)
\end{align*}
and $\varphi$ blows up as $|x|\to +\infty$.
As a byproduct, we can infer that, if $(\psi_k)\subset C_b(\Rd)$ is a bounded sequence, converging locally uniformly on $\Rd$ to some
function $\psi$, then $T(\cdot)\psi_n$ converges to $T(\cdot)\psi$ uniformly in $[0,+\infty)\times B_R$ for every $R>0$.
Indeed, if $t>0$, then we can estimate
\begin{align*}
|(T(t)(\psi_k\!-\!\psi))(x)|\le &\int_{B_r}|\psi_k(y)\!-\!\psi(y)|p(t,x,dy)\!+\!\int_{\Rd\setminus B_r}|\psi_k(y)\!-\!\psi(y)|p(t,x,dy)\\
\le &\|\psi_k\!-\!\psi\|_{C_b(B_r)}+2\sup_{k\in\N}\|\psi_k\|_{\infty}\sup_{t>0}\sup_{x\in B_R}p(t,x,\Rd\setminus B_r)
\end{align*}
for every $r,R,t>0$, $k\in\N$ and $x\in\Rd$.
Letting $k$ tend  to $+\infty$ we obtain that
\begin{eqnarray*}
\limsup_{k\to +\infty}\|T(t)(\psi_k-\psi)\|_{C_b((0,+\infty)\times B_R)}\le 2\sup_{k\in\N}\|\psi_k\|_{\infty}\sup_{t>0}\sup_{x\in B_R}p(t,x,\Rd\setminus B_r)
\end{eqnarray*}
for every $r,R>0$. Finally, letting $r$ tend to $+\infty$, we conclude that
\begin{eqnarray*}
\limsup_{k\to +\infty}\|T(t)(\psi_n-\psi)\|_{C_b((0,+\infty)\times B_R)}\le 0
\end{eqnarray*}
and we are done.

Coming back to the sequence $(\T(t)\f_{n_k})$, we observe that, by \eqref{vector-scalar}, we can estimate
$|\T(t)\f_{n_k}-\T(t-t_0)\g|^2\le T(t-t_0)(|\T(t_0)\f_{n_k}-\g|^2)$
in $\Rd$ for every $t>t_0$ and $k\in\N$. The above result, with $\psi_k=
|\T(t_0)\f_{n_k}-\g|$, yields immediately the assertion.

The same arguments can be used to prove the last part of the assertion.
\end{proof}

\begin{prop}\label{posi}
Under Hypotheses $\ref{hyp-base}(i)$-$(v)$, the semigroup $\T(t)$ is positive, in the sense that,
if $\f\in C_b(\Rd;\Rm)$ has all nonnegative components, then the function $\T(t)\f$ has
nonnegative components as well, for any $t>0$. If in addition $f_k$ does not identically vanish in $\Rd$ for some $k=1,\ldots,m$,
then $(\T(t)\f)_k>0$ in $\Rd$ for any $t>0$.
Consequently, the measures $p_{ij}(t,x,dy)$ $(i\neq j=1,\dots, m)$ in \eqref{for_rep_1} are nonnegative
and the measures $p_{ii}(t,x,dy)$ $(i=1,\ldots,m)$ are positive for any $t>0$ and $x \in \Rd$.
\end{prop}
\begin{proof}
To prove the first statement, we take advantage of a result in \cite{Ots88Ont}, which deals with the positivity of the semigroup in the case of bounded coefficients.
For this purpose, for any $n \in \N$, we introduce a smooth function $\psi_n:\R\to \R$ such that
\begin{eqnarray*}
\psi_n(x)=\left\{\begin{array}{ll}
x \qquad\;\, &|x|\le n,\\
n+1\qquad\;\, &x \ge n+1,\\
-n-1\qquad\;\, &x \le -n-1,
\end{array}\right.
\end{eqnarray*}
and set $\Psi_n(x)=(\psi_n(x_1),\ldots, \psi_n(x_d))$ for any $x\in\Rd$.
Let $\A_n$ be the elliptic operator defined as the operator $\A$ in \eqref{picco} with $q_{ij}, b_i$
and $C$ being replaced by $q_{ij,n}= q_{ij}\circ\Psi_n$, $b_{i,n}= b_i\circ\Psi_n$
and $C_n$ with entries $c_{hk,n}= c_{hk}\circ\Psi_n$, ($i,j=1,\ldots,d$, $h,k=1, \ldots, m$). Clearly, $\langle C_n(x)y,y\rangle$ is nonpositive  for any $x\in\Rd$ and $y\in\Rm$; thus, again \cite{DelLor11OnA} shows that, for any $\f\in C_b(\Rd;\Rm)$, the Cauchy
problem \eqref{pioggia} with $\A$ being replaced by $\A_n$ admits a unique classical solution $\uu_n$ which is bounded in $[0,+\infty)\times \Rd$.
Denote by $\T_n(t)$ the semigroup of contractions associated in $C_b(\Rd;\Rm)$ to the operator $\A_n$.
Since the off-diagonal entries of $C_n$ are nonnegative, \cite[Theorem 1.2]{Ots88Ont} implies that, if all the components of $\f$ are nonnegative (as we assume from now on),
then the components of $\uu_n$ are all nonnegative as well.
The interior Schauder estimates in Theorem \ref{thm-A2}, Arzel\`a-Ascoli theorem and a diagonal argument imply that there exists a subsequence $(\uu_{n_k})$ which converges to a function $\vv$ in $C^{1,2}([\varepsilon, T]\times K;\Rm)$ for any $0<\varepsilon<T$ and any compact set $K\subset\Rd$. Clearly, $\vv\in C^{1+\alpha/2,2+\alpha}_{\rm loc}((0,+\infty)\times\Rd;\Rm)$, satisfies the differential equation in \eqref{pioggia} and its components are all nonnegative in $(0,+\infty)\times\Rd$. We claim
 that $\vv$ can be extended by continuity to $\{0\}\times \Rd$, by setting $\vv(0,\cdot)=\f$. For this purpose,
we fix $R>0$ and let $\vartheta$ be a smooth cut-off function such that $\chi_{B_{R-1}}\le \vartheta \le \chi_{B_R}$. The function $\vv_k:=\vartheta \uu_{n_k}$ is bounded and continuous, and for $n_k> R$ it solves the Cauchy problem
 \begin{equation*}
\left\{
\begin{array}{ll}
D_t\vv_k=\A\vv_k+\g_k, & {\rm in}~(0,T]\times B_R,\\[1mm]
\vv_k(t,x)=0, & {\rm in}~ (0,T]\times \partial B_R,\\[1mm]
\vv_k(0,\cdot)=\vartheta\f, & {\rm in}~ B_R,
\end{array}
\right.
\end{equation*}
 where $\g_k= -{\rm Tr}(QD^2\vartheta)\uu_{n_k}-2(J_x\uu_{n_k})Q\nabla\vartheta-\uu_{n_k}\langle {\bf b},\nabla \vartheta\rangle$. Note that $\g_k$ is continuous in $(0, T)\times B_R$ and $\sqrt{t}\|\g_k(t,\cdot)\|_\infty\le C\|\f\|_\infty$ for every $t \in (0,1]$ and some positive constant $C$ independent of $k$ (see Proposition \ref{thm-A2}). Moreover,
 by the variation-of-constants-formula, we can write
\begin{eqnarray*}
\vv_k(t,x)=(\T_R(t)(\vartheta\f))(x)+\int_0^t(\T_R(t-s)\g_k(s,\cdot))(x)ds,\qquad\;\,t\in (0,T],\;\,x\in B_R,
\end{eqnarray*}
where $\T_R(t)$ is the semigroup generated by the realization of $\A$ in $C_b(B_R;\Rm)$ with homogeneous Dirichlet boundary conditions.
Since $\vv_k \equiv \uu_{n_k}$ in $[0,+\infty)\times B_{R-1}$ we deduce that
\begin{eqnarray*}
|\uu_{n_k}(t,x)-\f(x)|\le |(\T_R(t)(\vartheta\f))(x)-\f(x)|+c\sqrt{t}\|\f\|_\infty, \qquad\;\,t \in (0,1),\;\, x \in B_{R-1}.
\end{eqnarray*}
Letting first $k$ tend to $+\infty$ and, then, $t$ tend to $0^+$ we conclude that $\vv$ can be extended by continuity on $\{0\}\times B_{R-1}$ by setting $\vv(0,\cdot)=\f$.
By the arbitrariness of $R$ we conclude that $\vv$ can be extended by continuity to $[0,+\infty)\times \Rd$ by setting $\vv(0,\cdot)=\f$ in $\Rd$.
Hence, $\vv$ is a bounded classical solution to the Cauchy problem \eqref{pioggia}. By the uniqueness of the solution to this problem, it follows that $\vv \equiv\T(\cdot)\f$.
Thus, we conclude that all the components of $\T(t)\f$ are nonnegative.

Let us now suppose that $f_k\ge 0$ does not identically vanish in $\Rd$ for some $k$.
From Hypotheses \ref{hyp-base}(v) and the first part of the proof
it follows that $D_t(\T(\cdot)\f)_k\ge {\mathcal A}_k (\T(\cdot)\f)_k$ in $(0,+\infty)\times\Rd$, where
$\mathcal{A}_k= {\rm Tr}(Q D^2)+\langle b,\nabla \rangle+c_{kk}$.
Since $c_{kk}\le 0$, Hypothesis \ref{hyp-base}(iv) yields ${\mathcal A}_k\varphi\le a-c\varphi$ in $\Rd$.
As a byproduct, a generalized version of the classical maximum principle implies that
$(\T(\cdot)\f)_k\ge S_k(\cdot)f_k$ in $(0,+\infty)\times\Rd$,
where $S_k(t)$ is the semigroup associated to the realization of $\mathcal{A}_k$ in $C_b(\Rd)$ (see \cite[Chapter 3]{newbook}).
By Remark \ref{rmk-2.5}, $S_k(t)f_k$ is strictly positive in $\Rd$ for any $t>0$ so that
$(\T(\cdot)\f)_k$ is positive in $(0,+\infty)\times\Rd$.

Finally, let $A$ be any Borel subset of $\Rd$ and fix $h\in\{1,\ldots,m\}$. Then, the function
$\f=\chi_Ae_h$ belongs to $B_b(\Rd;\Rm)$ and there exists a bounded sequence $(\f_n)\subset C_b(\Rd;\Rm)$ which converges to
$\f$ pointwise almost everywhere with respect to the Lebesgue measure and, hence, with respect to each measure $p_{ij}(t,x,dy)$ ($i,j=1,\ldots,m$).
Moreover, without loss of generality, we can assume that $f_{n,i}\ge 0$ for any $n\in\N$ and $i=1,\ldots,m$. Combining the above facts we conclude that
\begin{eqnarray*}
p_{ih}(t,x,A)=(\T(t)\f(x))_i=\lim_{n\to +\infty}(\T(t)\f_n(x))_i\ge 0
\end{eqnarray*}
for any $t>0$, $x\in\Rd$, $i=1,\ldots,m$. Hence, the measure $p_{ih}(t,x,dy)$ is nonnegative for any $i=1,\ldots,m$. Taking $A=\R^d$, we also deduce that
$p_{hh}(t,x,\Rd)>0$ and the arbitrariness of $h$ allows us to conclude.
\end{proof}

\section{Systems of invariant measures}
\label{sect:3}
\begin{defi}
A family of positive and finite Borel measures $\{\mu_i: i=1,\ldots,m\}$ over $\Rd$ is a system of invariant measures for the semigroup $\T(t)$, if
for any $\f\in C_b(\Rd;\Rm)$ it holds that
\begin{align}
\sum_{i=1}^m\int_{\Rd}(\T(t)\f)_id\mu_i=\sum_{i=1}^m\int_{\Rd}f_id\mu_i.
\label{caricatore}
\end{align}
\end{defi}

To begin with, we characterize the set of all the fixed point of the semigroup $\T(t)$, i.e., the set
\begin{equation*}
{\mathcal E}=\{\f\in C_b(\Rd;\Rm):\T(t)\f\equiv \f\textrm{ for any }t\geq0\}.
\end{equation*}
Hypothesis \ref{hyp-base}(vi) yields that $\mathcal E$ is not empty since it contains the function $\f_0\equiv \xi$.
Hypothesis \ref{hyp-base}(vii) simply requires that $m$ is the \emph{minimum coupling order} in the sense that the system \eqref{pioggia} does not contain any lower order system that decouples. Such assumption, which is not restrictive, allows us to prove some properties of $\mathcal E$ which yield to assume that $\xi$ has all positive components and to deduce, as a consequence, that ${\mathcal E}$ is a one dimensional vector space spanned by the function $\f_0$.

\begin{prop}
${\mathcal E}$ is a one-dimensional vector space of constant functions spanned by the vector $\xi$.
\label{laurea}
\end{prop}

\begin{proof}
We split the proof into three steps.

{\em Step 1}. Here, we prove that $\f$ belongs to ${\mathcal E}$ if and only if $\f\equiv\eta$ for some $\eta\in\bigcap_{x\in\Rd}{\rm Ker}(C(x))$.

It is straightforward to check that, if $\f\equiv\eta$ for some $\eta\in\bigcap_{x\in\Rd}{\rm Ker}(C(x))$, then $\f$ belongs to ${\mathcal E}$.
Vice versa, let $\f$ belong to ${\mathcal E}$. From \eqref{vector-scalar} it follows that
$|\f|^2=|\T(t)\f|^2\leq T(t)|\f|^2$ in $\Rd$ for any $t\in (0,+\infty)$.
The above inequality and the invariance property of $\mu$, which yields
\begin{eqnarray*}
\int_{\Rd}(T(t)|\f|^2-|\f|^2)d\mu=0,\qquad\;\,t>0,
\end{eqnarray*}
imply that, for any $t>0$, $T(t)|\f|^2=|\f|^2$ $\mu$- almost everywhere.  Since $\mu$ is equivalent to the Lebesgue measure and the functions $|\f|^2$ and $T(t)|\f|^2$ are continuous in $\Rd$, we deduce that $|\T(t)\f|^2=T(t)|\f|^2=|\f|^2$ in $\Rd$ for any $t>0$.
Based on this result, we can now prove that $\f$ is constant. To this aim, we observe that the equality $\T(\cdot)\f=\f$ implies that $\f\in C^2(\Rd;\R^m)$ and, consequently, $|\f|^2\in C^2(\Rd)$.
Moreover, since $T(\cdot)|\f|^2$ is independent of $t$ and solves the equation $D_tu={\mathcal A}u$ in $(0,+\infty)\times\Rd$, it follows that $0=
{\mathcal A}T(\cdot)|\f|^2={\mathcal A}|\f|^2={\mathcal A}|\T(t)\f|^2$.
Using this fact and Hypothesis \ref{hyp-base}(iii) we deduce that
\begin{align*}
0& = D_t|\f|^2= D_t|\T(\cdot)\f|^2
= 2\langle \T(\cdot)\f,D_t\T(\cdot)\f\rangle \\
& = {\mathcal A}|\T(\cdot)\f|^2-2\sum_{i=1}^m|Q^{\frac{1}{2}}\nabla_x(\T(\cdot)\f)_i|^2
+2\langle C\T(\cdot)\f,\T(\cdot)\f\rangle \\
& \le -2\sum_{i=1}^m|Q^{\frac{1}{2}}\nabla_x(\T(\cdot)\f)_i|^2.
\end{align*}
Thus, taking Hypothesis \ref{hyp-base}(ii) into account we immediately get $J_x\f= J_x \T(t)\f=0$ for any $t>0$. Hence, $\f\equiv\eta$ for some $\eta\in\Rm$. Clearly,
$C(x)\eta=0$ for any $x\in\Rd$ so that $\eta\in\bigcap_{x\in\Rd}{\rm Ker}(C(x))$.

{\em Step 2}. Here, we prove that, if $\eta=(\eta_1,\ldots,\eta_m)$ belongs to $\bigcap_{x\in\Rd}{\rm Ker}(C(x))$, then
$\widehat\eta:=(|\eta_1|,\ldots,|\eta_m|)$ belongs to $\bigcap_{x\in\Rd}{\rm Ker}(C(x))$ as well.

To this aim, let $\g\equiv\eta$, $\widehat \g\equiv\widehat\eta$
and assume that $\eta\in\bigcap_{x\in\Rd}{\rm Ker}(C(x))$. By Step 1, $\g$ belongs to $\mathcal E$. Moreover, since $\T(\cdot)$ preserves positivity,
$|(\T(\cdot)\g)_j|\le (\T(\cdot)\widehat\g)_j$ in $[0,+\infty)\times\Rd$ for any $j=1,\ldots,m$.
Hence, taking \eqref{vector-scalar} into account, we get
\begin{align*}
|\widehat\g|^2=|\g|^2=|\T(t)\g|^2\leq |\T(t)\widehat\g|^2\le T(t)|\widehat\g|^2=|\widehat\g|^2,\qquad\;\,t>0,
\end{align*}
and, consequently, $|\T(\cdot)\widehat\g|^2=|\widehat\g|^2$. The same argument as in Step 1 implies that $\widehat\g\in{\mathcal E}$.
%Indeed, as we have already observed, $\widehat g_j=|g_j|=|(\T(\cdot)\g)_j|\le (\T(\cdot)\widehat\g)_j$ for any
%$j=1,\ldots,m$. If it there existed $k\in\{1,\ldots,m\}$
%such that $\widehat g_k< ((\T(\cdot)\widehat\g)(t_0,x_0))_k$ at some points $(t_0,x_0)\in (0,+\infty)\times \Rd$, then we would get $|\widehat\g|<|\T(t)\widehat\g|$ which leads to a contradiction.

{\em Step 3.} Now, we complete the proof. We claim that
the entries of any vector $\eta\in \Rm\setminus\{0\}$ belonging to $\bigcap_{x\in\Rd}{\rm Ker}(C(x))$ are all positive or all negative.
Fix any such vector $\eta$. In view of Step 2, the vector $\widehat\eta$ has all the components which are nonnegative.
By contradiction, assume that there exists $K\subset\{1,\ldots, m\}$ with $1\le |K|<m$ such that $\eta_i=0$ for any $i\in K$. Since $C(x)\widehat\eta=0$ for any $x\in\Rd$, it follows that
$\sum_{j\notin K}c_{ij}(x)|\eta_j|=0$ for any $i=1, \ldots, m$ and $x\in \Rd$. In particular, choosing $i \in K$ and recalling that the off-diagonal entries of the matrix $C(x)$
are nonnegative for any $x\in\Rd$, we conclude that
$c_{ij}\equiv 0$ for any $i\in K$ and $j\notin K$ contradicting Hypothesis \ref{hyp-base}(vii).
Now, we are almost done. Up to replacing $\eta$ with $-\eta$, we can assume that $\eta_1>0$. Then, all the other components are positive as well.
Indeed, if this were not the case, the nontrivial vector $\eta+\widehat\eta$, which belongs to $\bigcap_{x\in\Rd}{\rm Ker}(C(x))$,
would have at least one trivial component, which can not be the case.
It is straightforward to check that, if a subspace of $\Rm$ consists of vectors
whose entries are all positive or negative, then it is one-dimensional.
\end{proof}

\begin{rmk}
\label{rmk-102}
{\rm In view of Steps 1 and 2 in the proof of Proposition \ref{laurea} in the rest of this paper, we assume that all the entries of the vector $\xi$ are positive
and $|\xi|=1$.}
\end{rmk}

\begin{rmk}
{\rm In the particular case when the matrix $C(x)$ is irreducible for some $x\in\Rd$, the proof of
Proposition \ref{laurea} can be considerably simplified. Indeed, the Perron-Frobenius Theorem applied to the matrix $C(x)+\lambda I$, where
$\lambda$ is any real number greater than the maximum of the moduli of the negative eigenvalues of $C(x)$ and the moduli of the elements $c_{ii}(x)$ ($i=1,\ldots,m$),
shows that the kernel of $C(x)$ is one-dimensional and spanned by a vector $\xi$ whose components are all positive.

However, we stress that our assumptions on $C$ in general do not ensure that $C(x)$ is irreducible for some $x\in\Rd$.
For instance, suppose that
\begin{align*}
C(x)=
\begin{pmatrix}
-f(x)-h(x) & f(x) & h(x) \\
f(x) & -f(x)-g(x) & g(x) \\
h(x) & g(x) & -g(x)-h(x)
\end{pmatrix},
\end{align*}
for any $x\in\Rd$ and some smooth, nonnegative and nontrivial functions $f,g,h:\Rd\to\R$ compactly supported, respectively, in
$B_1$, $3e_1+B_1$ and $6e_1+B_1$.
It is easy to show that $C$ satisfies Hypotheses \ref{hyp-base}(iii), (v)-(vii) but it is
irreducible for no values of $x$, since for any $x\in\Rd$ at least one row of $C(x)$
vanishes.}
\end{rmk}

Now, we prove that our standing assumptions guarantee the existence of systems of invariant measures for $\T(t)$.

\begin{thm}\label{exi_inv_meas}
There exist infinitely many  systems of invariant measures for the semigroup $\T(t)$.
More precisely, if $\{\mu_j: j=1,\ldots,m\}$ is a system of invariant measures for $\T(t)$, then there exists a positive
constant $c$ such that $\mu_j=c\xi_j\mu$ for any $j=1,\ldots,m$, where $\mu$ is the invariant measure of the semigroup $T(t)$.
\end{thm}

In the proof of Theorem \ref{exi_inv_meas} we shall make use of the following result.

\begin{prop}\label{pre-res}
Under Hypotheses $\ref{hyp-base}(i)$-$(v)$, let $\boldsymbol{\mathcal R}_n$ $(n\in\N)$ be the operator defined on $C_b(\Rd;\R^m)$ by
\begin{eqnarray*}
\boldsymbol{\mathcal R}_n\f=\frac{1}{n}\sum_{k=0}^{n-1}\T(k)\f,\qquad\;\,\f\in C_b(\Rd;\Rm).
\end{eqnarray*}
Then, for any $\f\in C_b(\Rd;\R^m)$, $\boldsymbol{\mathcal R}_n\f$ converges to $\boldsymbol{\mathcal P}\f$ locally uniformly on $\Rd$, as $n\to +\infty$, where
$\boldsymbol{\mathcal P}$ is a projection onto the kernel of the operator $I-\T(1)$.
\end{prop}

\begin{proof}
To begin with, we prove that, for any $\f\in C_b(\Rd;\R^m)$, there exists a subsequence $(\boldsymbol{\mathcal R}_{n_k}\f)$ converging locally uniformly
to a function which belongs to ${\rm Ker}(I-\T(1))$.
For this purpose, we fix any such function $\f$ and split
\begin{eqnarray*}
\boldsymbol{\mathcal R}_n\f=\frac{1}{n}\f+\T(1)\bigg (\frac{1}{n}\sum_{k=0}^{n-2}\T(k)\f\bigg )=
\frac{1}{n}\f+\T(1)\bigg (\frac{n-1}{n}\boldsymbol{\mathcal R}_{n-1}\f\bigg ).
\end{eqnarray*}
Since the sequence $(\boldsymbol{\mathcal R}_{n-1}\f)$ is bounded, by Proposition \ref{prop-rossana}
there exists a subsequence  $(\boldsymbol{\mathcal R}_{n_k}\f)$ which converges locally uniformly in $\Rd$ to a function $\g\in C_b(\Rd;\Rm)$.
Clearly, $\g$ belongs to the kernel of $I-\T(1)$. Indeed,
\begin{eqnarray*}
\boldsymbol{\mathcal R}_{n_k}\f-\T(1)\boldsymbol{\mathcal R}_{n_k}\f=\frac{1}{n_k}\sum_{j=0}^{n_k-1}(\T(j)\f-\T(j+1)\f)=\frac{1}{n_k}(\f-\T(n_k)\f).
\end{eqnarray*}
Letting $k$ tend to $+\infty$ and taking \eqref{tiburtina} into account, we conclude that
$\g-\T(1)\g=0$.

Actually, we prove that all the sequence $(\boldsymbol{\mathcal R}_n\f)$ converges to $\g$ locally uniformly in $\Rd$.
For this purpose, we split $\f=\g+(\f-\g)$. Since $\g\in {\rm Ker}(I-\T(1))$, $\boldsymbol{\mathcal R}_n\g=\g$ for any $n\in\N$, so that, trivially, $\boldsymbol{\mathcal R}_n\g$ converges
uniformly in $\Rd$.
As far as the function $\f-\g$ is concerned, we first observe that $\f-\g$ is the local uniform limit in $\Rd$ of a sequence of functions in  $(I-\T(1))(C_b(\Rd;\Rm))$.
Indeed,
\begin{align*}
\f-\g=&\f-\lim_{k\to+\infty}\boldsymbol{\mathcal R}_{n_k}\f
=\lim_{k\to+\infty}
\frac{1}{n_k}\sum_{j=1}^{n_k-1}(I-(\T(1))^j)\f\\\
=&\lim_{k\to+\infty}
\frac{1}{n_k}\sum_{j=1}^{n_k-1}(I-\T(1))\sum_{h=0}^{j-1}\T(h)\f
=:\lim_{k\to+\infty}(I-\T(1))\bm\zeta_k,
\end{align*}
where each function $\bm\zeta_k$ belongs to $C_b(\Rd;\Rm)$ and all the limits appearing in the previous chain of equalities are local uniform in $\Rd$.
Now, we observe that
\begin{eqnarray*}
\boldsymbol{\mathcal R}_n(I-\T(1))\bm\zeta_k=\frac{1}{n}\bigg (\sum_{j=0}^{n-1}\T(j)\bm\zeta_k-\sum_{j=0}^{n-1}\T(j+1)\bm\zeta_k\bigg )
=\frac{1}{n}(\bm\zeta_k-\T(n)\bm\zeta_k).
\end{eqnarray*}
Hence,
\begin{align}
&\|\boldsymbol{\mathcal R}_n(\f-\g)\|_{C_b(B_r;\R^m)}\notag\\
\le &\|\boldsymbol{\mathcal R}_n[\f-\g-(I-\T(1))\bm\zeta_k]\|_{C_b(B_r;\R^m)}+\|\boldsymbol{\mathcal R}_n(I-\T(1))\bm\zeta_k\|_{C_b(B_r;\R^m)}\notag\\
\le& \|\T(\cdot)[\f-\g-(I-\T(1))\bm\zeta_k]\|_{C_b((0,+\infty)\times B_r;\R^m)}+
\frac{1}{n}\|\bm\zeta_k-\T(n)\bm\zeta_k\|_{\infty},
\label{frecciarossa}
\end{align}
for any $k,n\in\N$ and $r>0$.
Since $\f-\g-(I-\T(1))\bm\zeta_k$ vanishes locally uniformly in $\Rd$ as $k\to +\infty$, Proposition \ref{prop-rossana} shows that
$\T(\cdot)[\f-\g-(I-\T(1))\bm\zeta_k]$ converges uniformly in $(0,+\infty)\times B_r$ to zero as $k\to +\infty$.
Hence, letting first $n$ and then $k$ tend to $+\infty$ in the first and last side of \eqref{frecciarossa} we conclude that
$\boldsymbol{\mathcal R}_n(\f-\g)$ converges to zero locally uniformly in $\Rd$.

Let us denote by $\boldsymbol{\mathcal P} \f$ the limit of $\boldsymbol{\mathcal R}_n \f$ as $n \to +\infty$.
By the first part of the proof, we already know that the image of $\boldsymbol{\mathcal P}$ coincides with the kernel of $I-\T(1)$ and that
$\boldsymbol{\mathcal R}_n(\boldsymbol{\mathcal P}\f)=\boldsymbol{\mathcal P}\f$ for any $\f\in C_b(\Rd;\R^m)$ and $n\in\N$.
Thus, letting $n\to +\infty$ we deduce that $\boldsymbol{\mathcal P}^2=\boldsymbol{\mathcal P}$.
\end{proof}

\begin{proof}[Proof of Theorem $\ref{exi_inv_meas}$] We split the proof into three steps.

\emph{Step 1.} Here, we show that the family of measures $\{\nu_j: j=1,\ldots,m\}$ defined by $\nu_j=\xi_j\mu$ is a system of invariant measures for $\T(t)$.
For this purpose, we fix $\f\in C_b(\Rd;\R^m)$. Taking Lemma \ref{carusorompi} into account, it is easy to show that the function
$v=\langle \T(\cdot)\f,\xi\rangle$ is a bounded classical solution to the Cauchy problem
\begin{eqnarray*}
\left\{
\begin{array}{lll}
D_tv(t,x)={\mathcal A} v(t,x), &t\in (0,+\infty), &x\in\Rd,\\[1mm]
v(0,x)=\langle \f(x),\xi\rangle, &&x\in\Rd.
\end{array}
\right.
\end{eqnarray*}
Since this problem admits a unique bounded classical solution, it follows that
$v=T(\cdot)\langle \f,\xi\rangle$, i.e.,
$\langle (\T(t)\f)(x),\xi\rangle
=(T(\cdot)\langle\f,\xi\rangle)(x)$ for any $t\ge 0$ and $x\in\Rd$.

For any $j=1,\ldots,m$, let us set
$\nu_j=\xi_j\mu$. Then, the above result and the invariance of $\mu$ imply that
\begin{align*}
\sum_{j=1}^m\int_{\Rd}(\T(t)\f)_jd\nu_j
= &\int_{\Rd}T(t)\langle \f,\xi\rangle  d\mu=\int_{\Rd}\langle\f,\xi\rangle d\mu =\sum_{j=1}^m\int_{\Rd}f_jd\nu_j.
\end{align*}
Hence, the family $\{\nu_j: j=1,\ldots,m\}$ is a system of invariant measures for $\T(t)$.

\emph{Step 2.} Here, we prove that
\begin{equation}
\lim_{t\to +\infty}\boldsymbol{\mathcal P}_t\f:=\lim_{t\to +\infty}\frac{1}{t}\int_0^t(\T(s)\f)(\cdot)ds
=\bigg (\sum_{j=1}^m\int_{\Rd}f_jd\nu_j\bigg )\xi,
\label{marito}
\end{equation}
locally uniformly on $\Rd$ for any $\f\in C_b(\Rd;\R^m)$.
First of all we observe that
\begin{align}
(\boldsymbol{\mathcal P}_{t}\f)(x)=& \frac{1}{t}\int_0^{[t]}(\T(s)\f )(x) ds+ \frac{1}{t}\int_0^{\{t\}}(\T(s+[t])\f)(x) ds\notag\\
=& \frac{1}{t}\sum_{k=0}^{[t]-1}\int_0^{1}(\T(s+k)\f )(x) ds+ \frac{1}{t}\int_0^{\{t\}}(\T(s+[t])\f)(x) ds\notag\\
= & \frac{[t]}{t}(\boldsymbol{\mathcal R}_{[t]}\boldsymbol{\mathcal P}_1\f)(x)+ \frac{\{t\}}{t}((\T(1))^{[t]}\boldsymbol{\mathcal P}_{\{t\}}\f)(x)
\label{preduale}
\end{align}
for any $t>1$, $x\in\Rd$ and $\f\in C_b(\Rd;\R^m)$,
where $[t]$ and $\{t\}$ denote respectively the integer and the fractional part of $t$.
Taking  Proposition \ref{pre-res} into account and observing that
$\|(\T(1))^{[t]}\boldsymbol{\mathcal P}_{\{t\}}\f\|_{\infty}\le \|\f\|_{\infty}$ for any $t>0$,
from \eqref{preduale} we conclude that $\boldsymbol{\mathcal P}_t\f$ converges to $\boldsymbol{\mathcal P}\boldsymbol{\mathcal P}_1\f=:\boldsymbol{\mathcal P}_*\f$ locally uniformly on $\Rd$, as $t\to +\infty$, for every $\f\in C_b(\Rd;\R^m)$.

To show that $\boldsymbol{\mathcal P}_*$ is a projection it is enough to prove that
\begin{align}
\T(r)\circ \boldsymbol{\mathcal P}_*=\boldsymbol{\mathcal P}_*, \qquad\;\, r\geq0.
\label{cucina}
\end{align}
Indeed, once \eqref{cucina} is proved, we get $\boldsymbol{\mathcal P}_1\circ \boldsymbol{\mathcal P}_*=\boldsymbol{\mathcal P}_*$ and $\boldsymbol{\mathcal P}\circ\boldsymbol{\mathcal P}_*=\boldsymbol{\mathcal P}_*$, which clearly
imply that $\boldsymbol{\mathcal P}_*^2=\boldsymbol{\mathcal P}_*$. Fix $r>0$ and observe that, for any $t>0$, it holds that
\begin{align*}
\T(r)\boldsymbol{\mathcal P}_t\f
=&\frac{1}{t}\int_r^t(\T(s)\f)(\cdot)ds+\frac{1}{t}\int_t^{t+r}(\T(s)\f)(\cdot)ds \\
= & \boldsymbol{\mathcal P}_t\f+\frac{1}{t}\int_0^r((\T(t)-I)\T(s)\f)(\cdot)ds.
\end{align*}
Letting $t\rightarrow+\infty$ and taking Proposition \ref{prop-rossana} into account, we get \eqref{cucina}. Finally, we prove that $\boldsymbol{\mathcal P}_*$ is a projection on $\mathcal E$. This is equivalent to showing
that $\f\in\mathcal E$ if and only if $\boldsymbol{\mathcal P}_*\f=\f$. So, let us fix $\f\in {\mathcal E}$. Then, $\boldsymbol{\mathcal P}_t\f=\f$ for any $t>0$ and,
therefore, $\boldsymbol{\mathcal P}_*\f=\f$. Conversely, let us assume that $\boldsymbol{\mathcal P}_*\f=\f$; from \eqref{cucina} we deduce
that $\T(r)\f=\T(r)\boldsymbol{\mathcal P}_*\f=\boldsymbol{\mathcal P}_*\f=\f$ for any $r>0$, so that $\f\in {\mathcal E}$.

Since $\mathcal E$ consists of constant functions, we conclude that $\boldsymbol{\mathcal P}_*\f$ is a constant function for any $\f\in C_b(\Rd;\Rm)$. Hence, $\boldsymbol{\mathcal P}_*\f={\mathcal M}_{\f}\xi$
for any $\f\in C_b(\Rd;\Rm)$ and some bounded linear operator $\f\mapsto {\mathcal M}_{\f}\in\R$.
By the Riesz representation theorem, there exists a family $\{\mu_i: i=1,\ldots,m\}$ of finite and nonnegative Borel measures on $\Rd$ such that
\begin{equation}
{\mathcal M}_{\f}=\sum_{k=1}^m \int_{\Rd}f_k d\mu_k,\qquad\;\,\f\in C_0(\Rd;\R^m).
\label{montezuma}
\end{equation}

We claim that the previous formula can be extended to any function belonging to $C_b(\Rd; \Rm)$. To this aim, first of all we observe that
${\mathcal M}_{\f}$ is well defined for any $\f\in C_b(\Rd;\R^m)$ and the operator $C_b(\Rd;\R^m)\ni\f\mapsto {\mathcal M}_{\f}$ is bounded.
 Now, fix $\f\in C_b(\Rd;\R^m)$ and let $(\f_n)\subset C_0(\Rd; \Rm)$ be a sequence converging to $\f$ locally uniformly in $\Rd$ as $n \to +\infty$ and
 such that $\|\f_n\|_\infty\le \|\f\|_\infty$ for any $n \in \N$. Splitting $\boldsymbol{\mathcal P}_t\f=\boldsymbol{\mathcal P}_t\f_n+\boldsymbol{\mathcal P}_t(\f-\f_n)$ for any $t>0$ and $n\in\N$,
we can estimate
\begin{align*}
\bigg |\boldsymbol{\mathcal P}_t\f-\bigg (\sum_{k=1}^m \int_{\Rd}f_k d\mu_k\bigg )\xi\bigg |\le &|\boldsymbol{\mathcal P}_t\f_n-{\mathcal M}_{\f_n}\xi|+|\xi|\sum_{k=1}^m\int_{\Rd}|f_{n,k}-f_k|d\mu_k\\
&+\sup_{t>0}|\T(t)(\f_n-\f)|
\end{align*}
in $\Rd$ for any $t>0$ and $n\in\N$. Letting $t$ tend to $+\infty$ yields
\begin{align*}
&\limsup_{t\to+\infty}\bigg |\boldsymbol{\mathcal P}_t\f-\bigg (\sum_{k=1}^m \int_{\Rd}f_k d\mu_k\bigg )\xi\bigg |\\
\le &|\xi|\sum_{k=1}^m\int_{\Rd}|f_{n,k}-f_k|d\mu_k+\sup_{t>0}|(\T(t)(\f_n-\f))(x)|,
\end{align*}
for any $x\in\Rd$.
Finally, letting $n\to+\infty$ in the above estimate and taking again Proposition \ref{prop-rossana} into account, we conclude that
\eqref{montezuma} holds true also for any $\f\in C_b(\Rd;\Rm)$.

Now, we can complete the proof of \eqref{marito}.
%For this purpose, we begin by observing that
%\begin{align*}
%\xi_i\sum_{k=1}^m\int_{\Rd}(\T(t)\f)_kd\mu_k
%=&M_{\T(t)\f}\xi_i=(\boldsymbol{\mathcal P}_*\T(t)\f)_i\\
%= & (\boldsymbol{\mathcal P}_*\f)_i+\lim_{\tau\to +\infty}\frac{1}{\tau}\left(\int_{\tau}^{t+\tau}\!(\T(s)\f)_i(0)ds-\int_0^t(\T(s)\f)_i(0)ds\right) \\
%= & (\boldsymbol{\mathcal P}_*\f)_i=\xi_iM_{\f}=\xi_i\sum_{k=1}^m\int_{\Rd}f_kd\mu_k.
%\end{align*}
%for any $t>0$ and $\f\in C_b(\Rd;\Rm)$. So, formula \eqref{caricatore} follows for functions in $C_b(\Rd;\R^m)$, observing that $\xi_i>0$ for any $i=1,\ldots,m$.
%Moreover, if $\f\equiv\xi$, then $\boldsymbol{\mathcal P}_*\f\equiv\xi$, so that $M_{\f}=1$, or equivalently,
%\begin{equation}
%1= \sum_{i=1}^m\xi_i\mu_i(\Rd),\qquad\;\, i=1, \ldots, m.
%\label{angiulino}
%\end{equation}
%
Since $\{\nu_j: j=1,\ldots,m\}$ is a system of invariant measures for $\T(t)$, applying Fubini theorem, we easily deduce that
\begin{eqnarray*}
\sum_{j=1}^m\int_{\Rd}(\boldsymbol{\mathcal P}_t\f)_jd\nu_j=\sum_{j=1}^m\int_{\Rd}f_jd\nu_j
\end{eqnarray*}
for any $t>0$ and $\f\in C_b(\Rd;\R^m)$. Letting $t$ tend to $+\infty$ in the previous formula, by dominated convergence we can infer that
\begin{equation}
\sum_{j=1}^m\int_{\Rd}(\boldsymbol{\mathcal P}_*\f)_jd\nu_j=\sum_{j=1}^m\int_{\Rd}f_jd\nu_j
\label{ladri}
\end{equation}
or, equivalently, taking into account that $|\xi|=1$,
\begin{eqnarray*}
\sum_{j=1}^m\int_{\Rd}f_jd\mu_j=\sum_{j=1}^m\int_{\Rd}f_jd\nu_j
\end{eqnarray*}
for any $\f\in C_b(\Rd;\R^m)$. Taking $\f=fe_j$ with $f\in C_b(\Rd)$ and $j=1,\ldots,m$, we conclude that
\begin{eqnarray*}
\int_{\Rd}fd\mu_j=\int_{\Rd}fd\nu_j,\qquad\;\,f\in C_b(\Rd).
\end{eqnarray*}
This is enough to infer that $\mu_j=\nu_j$ for any $j=1,\ldots,m$.

{\emph Step 3.} Suppose that $\{\widetilde\mu_j: j=1,\ldots,m\}$ is another system of invariant
measures for $\T(t)$. Then, \eqref{ladri} can be written with $\nu_j$ being replaced by $\widetilde\mu_j$.
Letting $t$ tend to $+\infty$, we deduce that
\begin{eqnarray*}
\sum_{k=1}^m\xi_k\widetilde\mu_k(\Rd)\sum_{j=1}^m\int_{\Rd}f_jd\nu_j=\sum_{j=1}^m\int_{\Rd}f_jd\widetilde\mu_j
\end{eqnarray*}
for any $\f\in C_b(\Rd;\R^m)$. Hence,
$\widetilde\mu_j=c\xi_j\mu$ for any $j=1,\ldots,m$, where
$c=\sum_{k=1}^m\xi_k\widetilde\mu_k(\Rd)$. This completes the proof.
\end{proof}

\begin{rmk}
\label{rmk-sign}{\rm
The sign condition on the quadratic form induced by $C$ is inspired by the scalar case where typically one assumes that
the potential term of the elliptic operator identically vanishes on $\Rd$ to guarantee the existence of an invariant measure.
Hypothesis \ref{hyp-base}(iii) seems the natural extension in the multidimensional case. If that condition is violated, then we can find examples of matrix-valued functions $C$ such that nontrivial systems of invariant measures for the associated
semigroup $\T(t)$ do not exist.
Consider for instance the particular case when $C$ is a symmetric constant matrix
and assume that $\langle C\xi,\xi\rangle>0$ for some $\xi\in\R^m$.
In this case $\sigma(C)\cap \R^+\neq \varnothing$ and the ordinary differential equation
$D_t\uu=C\uu$ admits a solution $\uu$, with all positive components,
such that $|\uu(t)|\ge e^{t\lambda}$ for any $t>0$ and some $\lambda>0$.
Let us set $\uu_0=\uu(0)$. Then, clearly, $\uu=\T(\cdot)\uu_0$.
It thus follows that for any $t>0$ there exists $j_t\in\{1,\ldots,m\}$ such that $(\T(t)\uu_0)_{j_t}\ge m^{-1/2}e^{\lambda t}$.
Since $\T(\cdot)\uu_0$ is independent of $x$, the invariance property \eqref{caricatore} shows that
\begin{align*}
\sum_{j=1}^mu_{0,j}\mu_j(\Rd)=&\sum_{j=1}^m\int_{\Rd}u_{0,j}d\mu_j=\sum_{j=1}^m\int_{\Rd}(\T(t)\uu_0)_jd\mu_j
%=\sum_{j=1}^m(\T(t)\uu_0)_j\mu_j(\Rd)
\\
\ge &(\T(t)\uu_0)_{j_t}\mu_{j_t}(\Rd)
\ge m^{-\frac{1}{2}}\min_{1\le j\le m}\mu_j(\Rd)e^{\lambda t}
\end{align*}
for any $t>0$. Letting $t$ tend to $+\infty$ we get to a contradiction.

On the other hand, if $C$ is a matrix-valued function which satisfies the condition $\sup_{x\in\Rd,|\xi|=1}\langle C(x)\xi,\xi\rangle<0$,
then, by \cite[Theorem 2.6]{DelLor11OnA}, the sup-norm of the function $\T(t)\f$ exponentially decreases to zero as $t\to +\infty$
for any $\f\in C_b(\Rd;\R^m)$. Hence, if we take $\f=e_k$ ($k=1,\ldots,m$), then using again \eqref{caricatore} we obtain
\begin{align*}
\mu_k(\Rd)=\sum_{j=1}^m\int_{\Rd}(\T(t)e_k)_jd\mu_j
\end{align*}
and, letting $t$ tend to $+\infty$, by dominated convergence we conclude that $\mu_k(\Rd)=0$ for any $k\in\{1,\ldots,m\}$.
}\end{rmk}

\begin{example}
\label{example-1}
{\rm
Let $\A$ be as in \eqref{picco}
with $Q(x)=(1+|x|^2)^{\gamma}Q^0$ for any $x\in\Rd$, $Q^0$ being a constant, symmetric and positive definite $d\times d$-matrix and $\gamma$ being a nonnegative number.
Let further ${\bf b}(x)=-b_0x(1+|x|^2)^{\beta}$, for any $x\in\Rd$ and some positive constants $b_0$ and $\beta$, and $C$ be any $m\times m$-matrix, with entries in $C^{\alpha}_{\rm loc}(\Rd)$ for some $\alpha\in (0,1)$, and such that the elements on the main diagonal
are negative, whereas the off-diagonal ones are positive and the sum of the elements of each row and column is zero. 
By the Gershgorin circle theorem, applied to $C(x)+(C(x))^*$, (see \cite[Theorem 1.11]{gersh}), we can infer that $\langle C(x)y,y\rangle\le 0$ for any $x\in\Rd$ and $y\in\Rm$. Moreover, we can take as $\xi$ the vector with all entries equal to one (see Remark \ref{nosym}).
It is easy to check that if $\beta>(\gamma-1)^+$ then Hypothesis \ref{hyp-base}(iv) is satisfied as well, with $\varphi(x)=1+|x|^2$ for any $x\in \Rd$. Indeed
\begin{equation}\label{afi}
({\mathcal A} \varphi)(x)=2\varphi(x)[(1+|x|^2)^{\gamma-1}{\rm Tr}(Q^0)-b_0|x|^2(1+|x|^2)^{\beta-1}],\qquad\;\,x\in\Rd,
\end{equation}
and the term in brackets in \eqref{afi} tends to $-\infty$ as $|x|\to +\infty$. Therefore, we can determine two positive constants $a$ and $c$ such that ${\mathcal A} \varphi \le a -c \varphi$ in $\Rd$. In this case, all the assumptions in Theorem \ref{exi_inv_meas} are satisfied and consequently it can be applied.} \end{example}

\begin{rmk}
{\rm Theorem \ref{exi_inv_meas} shows that, in general, a system of invariant measures $\{\mu_i: i=1,\ldots,m\}$ for $\T(t)$ does not consist only of probability measures. We can infer that each $\mu_i$ is a probability measure if and only if $\xi_i=1$ for any $i=1,\ldots,m$.}
\end{rmk}

\subsection{The semigroup $\bm{\T(t)}$ in $\bm{L^p}$-spaces}
In this subsection, we prove that the semigroup $\T(t)$ can be extended, with a bounded strongly continuous semigroup, to the $L^p$-spaces related to any system of invariant measures $\{\mu_i: i=1,\ldots,m\}$ and we investigate on some of its smoothing effects in these spaces.

Throughout the section, $\{\mu_i: i=1,\ldots,m\}$ is any system of invariant measures for $\T(t)$.
Moreover, for any $p\in [1,+\infty)$, we write $L^p_{\bm\mu}(\Rd;\Rm)$ to denote the set $\bigotimes_{i=1}^mL^p_{\mu_i}(\Rd)$, which we endow with
the natural norm $\f\mapsto\left (\sum_{i=1}^m\int_{\Rd}|f_i|^pd\mu_i\right )^{1/p}$.
Similarly, by $W^{j,p}_{\bm\mu}(\Rd;\Rm)$ we denote the Sobolev space of order $j$ of all the functions $\f\in L^p_{\bm\mu}(\Rd;\Rm)$ whose distributional derivatives up to the $j$-th order are in $L^p_{\bm\mu}(\Rd;\Rm)$. It is normed by setting $\|\f\|_{W^{j,p}_{\bm\mu}(\Rd;\R^m)}=\sum_{|\alpha|\le j}\|D^{\alpha}\f\|_{L^p_{\bm\mu}(\Rd;\Rm)}$
for any $\f\in W^{j,p}_{\bm\mu}(\Rd;\Rm)$. To lighten the notation we write $\|\cdot\|_{p,{\bm\mu}}$, resp. $\|\cdot\|_{p,{\mu_i}}$, resp. $\|\cdot\|_{j,p,\bm\mu}$ in place of $\|\cdot\|_{L^p_{\bm\mu}(\Rd;\Rm)}$, resp. $\|\cdot\|_{L^p_{\mu_i}(\Rd)}$ and $\|\cdot\|_{W^{j,p}_{\bm\mu}(\Rd;\Rm)}$.

\begin{rmk}\label{rm_dens}{\rm
Since the measures $\mu_i$ ($i=1, \ldots,m$) are finite Borel measures, the space $C_b(\Rd;\Rm)$ is dense in $L^p_{\bm\mu}(\Rd;\Rm)$ for any $p \in[1,+\infty)$.
See \cite[Remark 1.46]{AFP} for further details.}
\end{rmk}

\begin{prop}
The semigroup $\T(t)$ extends to a strongly continuous semigroup $($still denoted by $\T(t))$
on $L^p_{\bm\mu}(\Rd;\Rm)$ for any $1\le p<+\infty$. Moreover for any $p\in[1,+\infty)$ it holds that
\begin{align}
\label{festa}
\|\T(t)\|_{\mathcal L(L^p_{\bm\mu}(\Rd;\Rm))}\leq 2^{\frac{p-1}{p}}, \qquad\;\, t>0.
\end{align}
Finally, the set
${\mathcal D}=\{\uu\in C_b(\Rd;\R^m)\cap\bigcap_{p<+\infty} W^{2,p}_{\rm loc}(\Rd;\R^m): \bm{\mathcal A}\uu\in C_b(\Rd;\R^m)\}$
is a core for the infinitesimal generator ${\bf A}_p$ of $\T(t)$ in $L^p_{\bm\mu}(\Rd;\R^m)$, for any $p\in [1,+\infty)$.
\end{prop}

\begin{proof}
Since $\|\T(t)e_k\|_{\infty}\leq 1$, it follows that $p_{ik}(t,x,\Rd)\leq 1$ for any $i,k=1,\ldots,m$, $t>0$ and $x\in\Rd$.
Thus, the Jensen inequality and formula \eqref{for_rep_1} yield
\begin{align*}
|(\T(t)\f)_i(x)|^p
\leq & 2^{p-1}\sum_{k=1}^m\left|\int_{\Rd}f_k(y)p_{ik}(t,x,dy)\right|^p \\
\leq & 2^{p-1}\sum_{k=1}^m [p_{ik}(t, x, \Rd)]^{p-1}\int_{\Rd}\left|f_k(y)\right|^pp_{ik}(t,x,dy) \\
\leq & 2^{p-1}(\T(t)(|f_1|^p, \ldots,|f_m|^p ))_i(x)
\end{align*}
for any $t>0, x\in \Rd$, $i=1, \ldots, m$, $\f\in C_b(\Rd;\R^m)$ and $p \in[1,+\infty)$. Moreover, by the invariance property \eqref{caricatore} we deduce that
\begin{align*}
\sum_{i=1}^m\int_{\Rd}|(\T(t)\f)_i|^pd\mu_i
\leq & 2^{p-1}\sum_{i=1}^m\int_{\Rd}(\T(t)(|f_1|^p, \ldots,|f_m|^p ))_id\mu_i \\
= & 2^{p-1}\sum_{i=1}^m\int_{\Rd}|f_i|^pd\mu_i
\end{align*}
for any $t>0$ and $\f\in C_b(\Rd;\Rm)$. Taking Remark \ref{rm_dens} into account, from the previous chain of inequalities we easily deduce that $\T(t)$ extends
to a linear bounded operator in $L^p_{\bm\mu}(\Rd;\Rm)$ and formula \eqref{festa} follows. The semigroup property easily follows. Hence, $\T(t)$ is a semigroup in
$L^p_{\bm\mu}(\Rd;\R^m)$.

To show that such a semigroup is strongly continuous, we first observe that $\|\T(t)\f-\f\|_{p,\bm{\mu}}$ vanishes as $t \to 0^+$ for any $\f\in C_b(\Rd;\Rm)$ and $p \in [1,+\infty)$. Indeed,
for any such function, $\T(t)\f$ converges locally uniformly to $\f$ as $t\to 0^+$ and $\|\T(t)\f\|_\infty\leq\|\f\|_\infty$; hence by the dominated convergence theorem we get the assertion.

Now, fix $\f\in L^p_{\bm\mu}(\Rd;\Rm)$ and let $(\f_n)$ be a sequence in $C_b(\Rd;\Rm)$ converging
 to $\f$ in $L^p_{\bm \mu}(\Rd;\Rm)$ as $n \to +\infty$. For any $i=1, \ldots,m$ and $n \in \N$, we can estimate
\begin{align*}
\|(\T(t)\f)_i-f_i\|_{p,\mu_i}
%\leq & \|\T(t)\f-\f\|_{p,\bm\mu}\\
\leq & \|\T(t)(\f-\f_n)\|_{p,\bm\mu}+\|\T(t)\f_n-\f_n\|_{p,\bm\mu}+\|\f_n-\f\|_{p,\bm\mu}\notag \\
\leq &\|\T(t)\f_n-\f_n\|_{p,\bm\mu}+(2^{\frac{p-1}{p}}+1)\|\f_n-\f\|^p_{p,\bm\mu},
\end{align*}
where in the last line we have used \eqref{festa}. Letting first $n$ tend to $+\infty$ and then $t$ tend to $0^+$, we deduce that
$\|(\T(t)\f)_i-f_i\|_{L^p_{\mu_i}(\Rd)}$ vanishes as $t\to 0^+$.

To conclude the proof, let us prove that ${\mathcal D}$ is a core for the infinitesimal generator ${\bf A}_p$ of $\T(t)$ in $L^p_{\bm\mu}(\R^d;\R^m)$ for any $p\in [1,+\infty)$.
For this purpose, we observe that, by \cite[Proposition 3.1]{DelLor11OnA}, ${\mathcal D}$ coincides with the set of all $\uu\in C_b(\Rd;\Rm)$ such that
$\sup_{t\in (0,1)}t^{-1}\|\T(t)\uu-\uu\|_{\infty}<+\infty$ and $\lim_{t\to 0^+}t^{-1}(\T(t)\uu-\uu)=\A\uu$ pointwise in $\Rd$.
Since all the $\mu_i$'s are finite measures, from the previous properties and dominated convergence we immediately deduce that, if $\uu\in {\mathcal D}$,
then $t^{-1}(\T(t)\uu-\uu)$ converges to $\A\uu$ in $L^p(\Rd;\Rm)$ as $t$ tends to $0^+$, for any $p\in [1,+\infty)$. Hence, $\uu\in D({\bf A}_p)$ and
${\bf A}_p\uu=\A\uu$, so that ${\mathcal D}\subset D({\bf A}_p)$. Moreover, since it contains $C^{\infty}_c(\Rd;\Rm)$, ${\mathcal D}$ is dense in $L^p_{\bm\mu}(\Rd;\Rm)$ (see also Remark \ref{rm_dens}).
Finally, by \cite[Proposition 3.2]{DelLor11OnA}, $\T(t)$ leaves ${\mathcal D}$ invariant. Hence, applying \cite[Proposition II.1.7]{EN}, we conclude that ${\mathcal D}$ is a core for ${\bf A}_p$.
\end{proof}

The characterization of the domain $D({\bf A}_p)$ is a very hard task and, as the scalar case reveals, $D({\bf A}_p)$ has been characterized only in rather particular situations.
Still the scalar case shows that, in general, we can not expect the semigroup $\T(t)$ to be analytic in $L^p_{\bm\mu}(\Rd;\R^m)$. We refer the interested reader to
\cite{newbook} for further details.
Nevertheless, $\T(t)$ has smoothing effects since maps $L^p_{\bm\mu}(\Rd;\Rm)$ into $W^{2,p}_{\bm\mu}(\Rd;\Rm)$. This property is a consequence of some
pointwise estimates for the first- and second-order spatial derivatives of the semigroup $\T(t)$.
More precisely, we provide sufficient conditions for the estimates
\begin{equation}\label{der_est}
|D^k_x\T(t)\f|^p \le \Gamma_{p,k,h}(t)T(t)\bigg(\sum_{j=0}^h|D^j\f|^2\bigg)^{\frac{p}{2}},
\end{equation}
to hold in $\Rd$ for any $\f\in C^h_b(\Rd,\R^m)$, $t>0$, $p>1$, $k\in\{1,2\}$ and $h\in\{0,\ldots,k\}$ where $\Gamma_{p,k,h}$ is a positive function defined in $(0,+\infty)$.
Estimates \eqref{der_est} also allow us to prove a partial characterization of $D({\bf A}_p)$ (see Corollary \ref{coro-3.12}).

To ease the notation, in the rest of this section we set
\begin{eqnarray*}
\begin{array}{lll}
&{\mathscr B}_{i,\uu}=\displaystyle\sum_{|\alpha|=i}\sum_{k=1}^m\langle J{\bf b}\nabla_xD^{\alpha}_xu_k,\nabla_xD^{\alpha}_xu_k\rangle,\quad\!\!
&{\mathscr B}_2=\displaystyle\bigg (\sum_{i,j=1}^d|D_{ij}{\bf b}|^2\bigg )^{\frac{1}{2}},\\
&{\mathscr C}_i=\displaystyle\bigg (\sum_{|\alpha|=i}|D^{\alpha}C|^2\bigg )^{\frac{1}{2}},
&{\mathscr F}_{i,\uu}=\displaystyle\bigg (\sum_{k=1}^m\sum_{|\alpha|=i}|Q^{\frac{1}{2}}\nabla_x D^{\alpha}_xu_k|^2\bigg )^{\frac{1}{2}},\\
&{\mathscr Q}_i=\displaystyle\bigg (\sum_{|\alpha|=i}|D^{\alpha}Q|^2\bigg )^{\frac{1}{2}},
\end{array}
\end{eqnarray*}
where, if $|\alpha|=0$, then $D^{\alpha}_xu_k=u_k$. We also denote by $r$ the ``best'' function which bounds from above the quadratic form associated to the Jacobian matrix of the drift ${\bf b}$, i.e.,
$\langle J{\bf b}(x)y,y\rangle\le r(x)|y|^2$ for any $x,y\in\Rd$.

Estimate \eqref{der_est} with $k=1$ has been already proved in \cite{AngLorPal} when the differential operator ${\bf\A}$ is in divergence form with first-order coupling. In our case, the assumptions considered in \cite{AngLorPal} and yielding \eqref{der_est} with $k=1$ force the first-order derivatives of the entries of the matrix $C$ to be bounded. To enlighten this hypothesis we consider a set of different assumptions.

\begin{hyp}\label{sandro}
\begin{enumerate}[\rm (i)]
\item
The coefficients $q_{ij}$, $b_i$ and the entries $c_{hk}$ of the matrix valued function $C$, belong to $C^{1+\alpha}_{\rm{loc}}(\Rd)$ for any $i,j=1,\ldots, d$ and $h,k=1, \ldots,m$;
\item
for any $p\in (1,+\infty)$, there exists a positive constant $c_p$ such that
\begin{equation*}
K_p:=\sup_{\Rd}\bigg(r+(1-p)\mu_Q+\frac{1}{4(p-1)\mu_Q}{\mathscr Q}_1^2+c_p{\mathscr C}_1^2\bigg )<+\infty,
\end{equation*}
\end{enumerate}
\end{hyp}

The proof of estimate \eqref{der_est} is the content of the forthcoming Theorems \ref{thm-avvvooooccato} and \ref{thm-4.4}. In the first theorem,
we consider the easiest case $k=1$, under the additional Hypotheses \ref{sandro}. Then, strengthening the assumptions on the coefficients of the operator $\A$,
we prove estimate \eqref{der_est} with $k=2$. Before entering into details,
we preliminary observe that it suffices to prove \eqref{der_est} for $p\in (1,2]$. Indeed, suppose that \eqref{der_est} hold true with $p=2$. Then,
for $p>2$, $h,k=0,1,2$, with $h\le k$, $t>0$ and $\f\in C^h_b(\Rd;\Rm)$, we can estimate
\begin{align*}
|D^k_x\T(t)\f|^p=(|D^k_x\T(t)\f|^2)^{\frac{p}{2}}\le &
\bigg (\Gamma_{2,k,h}(t)T(t)\sum_{j=0}^h|D^j\f|^2\bigg )^{\frac{p}{2}}\\
\le &(\Gamma_{2,k,h}(t))^{\frac{p}{2}}T(t)\bigg (\sum_{j=0}^h|D^j\f|^2\bigg )^{\frac{p}{2}}.
\end{align*}

We can also just consider functions in $C^{\infty}_c(\Rd;\Rm)$. Once \eqref{der_est} is proved for such smooth functions,
we can use a density argument to extend its validity to any $\f\in C^h_b(\Rd;\Rm)$. Indeed, for any $\f\in C^h_b(\Rd;\Rm)$ there exists a sequence
$(\f_n)\subset C^{\infty}_c(\Rd;\Rm)$, bounded with respect to the $C^h_b(\Rd;\Rm)$-norm and converging to $\f$ in $C^h_b(B_r;\Rm)$ for any $r>0$.
Writing \eqref{der_est} with $\f$ being replaced by $\f_n$ and using \cite[Proposition 3.2]{AALT}, we can let $n$ tend to $+\infty$ and obtain \eqref{der_est} in 
its full generality. 

Hence, in the proof of Theorems \ref{thm-avvvooooccato} and \ref{thm-4.4} we will assume that $p\in (1,2]$ and $\f\in C^{\infty}_c(\Rd;\Rm)$ without further mentioning it.

We finally observe that, for any $p\in (1,2]$, $A,B>0$, it holds that
\begin{align}
\min_{x>0}\left (Ax^{\frac{2}{p-2}}+Bx^{\frac{2}{p}}\right )=\left (\frac{2-p}{p}\right )^{\frac{p-2}{2}}\frac{2}{p}A^{1-\frac{p}{2}}B^{\frac{p}{2}}.
\label{maggiordomo}
\end{align}
\begin{thm}
\label{thm-avvvooooccato}
Under Hypotheses $\ref{sandro}$, estimate \eqref{der_est} holds true, with $k=1$ and $\Gamma_{p,1,h}(t)=\gamma_{h,p} e^{C_{h,p}t}t^{\frac{p}{2}(h-1)}$ for any $t>0$, where
$\gamma_{h,p}$ and $C_{h,p}$ are positive constants for any $h=0,1$.
\end{thm}

\begin{proof}
We fix $p$, $\f$ and consider the solution $\vv_n$ to problem \eqref{prob_approx_Neu} and the positive semigroup $T_n^{\mathcal N}(t)$
associated to the realization of $\mathcal{A}$ in $C_b(\overline B_n)$ with homogeneous Neumann
boundary conditions. We split the proof into two steps. In the first one we prove \eqref{der_est} with $h=1$, in the second one we deal with the case $h=0$.

{\it Step 1.} To prove \eqref{der_est} with $h=k=1$ it suffices to show that
\begin{equation}\label{stima_grad1_appr}
|J_x \vv_n(t,\cdot)|^p \le e^{C_{1,p}t}T_n^{\mathcal N}(t)(|\f|^2+|J\f|^2)^{\frac{p}{2}}
\end{equation}
for any $t>0$, $n\in \N$ and some positive constant $C_{1,p}$. Indeed, once \eqref{stima_grad1_appr} is proved, estimate \eqref{der_est} with $h=k=1$,
($\gamma_{1,p}=1$ and $C_{1,p}=C_p$) will follow simply letting $n\to +\infty$, recalling that, for any $t>0$, $T_n^{\mathcal N}(t)f$ converges to $T(t)f$, as $n\to +\infty$,
locally uniformly in $\Rd$ (see Section \ref{sect:2}).
So, let us prove \eqref{stima_grad1_appr}. For any $\varepsilon >0$ and $n\in \N$, let us consider the function
$w_n= (|\vv_n|^2+|J_x \vv_n|^2+\varepsilon)^{\frac{p}{2}}$.
From \cite[Chapter 7, Section 10]{LadSolUra68Lin} it follows that
$w_n \in C^{1,2}([0,+\infty)\times\overline B_n)\cap C_b([0,T]\times \overline B_n)$ for any $T>0$. Moreover, $w_n$ solves the problem
\begin{equation}\label{pro_scal}
\left\{\begin{array}{ll}
D_t w_n- \mathcal Aw_n=pw_n^{1-\frac{2}{p}}g, \quad\,\,  &{\rm in}~(0,+\infty)\times B_n,\\[1mm]
\displaystyle\frac{\partial w_n}{\partial \nu} \le 0,\quad\,\,  &{\rm on}~(0,+\infty)\times \partial B_n,\\[2.5mm]
w_n(0)= (|\f|^2+|J \f|^2+\varepsilon)^{\frac{p}{2}}, \quad\,\,  &{\rm in}\,~ B_n,
\end{array}
\right.
\end{equation}
where $g=\psi_1+\psi_2+(2-p)w_n^{-\frac{2}{p}}\psi_3$, with
\begin{align*}
\psi_1= &{\mathscr B}_{0,\vv_n}+ \langle C\vv_n,\vv_n \rangle+\sum_{i=1}^d\langle CD_i \vv_n,D_i \vv_n \rangle
-{\mathscr F}_{0,\vv_n}^2-{\mathscr F}_{1,\vv_n}^2,
\\[1mm]
\psi_2= &\sum_{i,j,h=1}^d\sum_{k=1}^m D_hq_{ij}D_{ij} v_{n,k}D_hv_{n,k}+\sum_{h=1}^d\sum_{j,k=1}^m D_hC_{jk}v_{n,k}D_h v_{n,j},\\[1mm]
\psi_3=&\sum_{i,j=1}^dq_{ij}\bigg (\langle\vv_n,D_i\vv_n\rangle\!+\!\sum_{h=1}^d\langle D_{ih}\vv_n,D_h\vv_n\rangle\bigg )\!
\bigg (\langle\vv_n,D_j\vv_n\rangle\!+\!\sum_{\ell=1}^m\langle D_{j\ell}\vv_n,D_{\ell}\vv_n\rangle\bigg )
\end{align*}
and the boundary condition in \eqref{pro_scal} follows since the normal derivative of
$|\nabla_x v_{n,k}|^2$ is nonpositive in $(0,+\infty)\times\partial B_n$
for any $k=1, \ldots,m$ (see e.g., \cite{BerFor04Gra}).

Using Hypothesis \ref{hyp-base}(iii), the Cauchy-Schwarz and the Young's inequalities, we estimate the functions $\psi_i$ ($i=1,2,3$) as follows:
\begin{align*}
\psi_1 \le& r|J_x \vv_n|^2-{\mathscr F}_{0,\vv_n}^2-{\mathscr F}_{1,\vv_n}^2,\\[1mm]
\psi_2 \le& {\mathscr Q}_1|J_x \vv_n||D^2_x\vv_n|+ {\mathscr C}_1|\vv_n||J_x \vv_n|
\\ \le& \frac{1}{4c_p}|\vv_n|^2+\bigg(\frac{1}{4(p-1)\mu_Q}{\mathscr Q}_1^2+c_p{\mathscr C}_1^2\bigg)|J_x \vv_n|^2
+(p-1)\mu_Q|D^2_x\vv_n|^2,\\[1mm]
\psi_3\le  &\bigg(\sum_{k=1}^m|v_{n,k}||Q^{\frac{1}{2}}\nabla_xv_{n,k}|\bigg)^2+2\sum_{h,k=1}^m|v_{n,h}||Q^{\frac{1}{2}}\nabla_xv_{n,h}|
\sum_{i=1}^d|D_iv_{n,k}||Q^{\frac{1}{2}}\nabla_xD_iv_{n,k}|\\
&+\sum_{i,j=1}^d\sum_{h,k=1}^m|D_iv_{n,h}||D_jv_{n,k}||Q^{\frac{1}{2}}\nabla_xD_iv_{n,h}||Q^{\frac{1}{2}}\nabla_xD_jv_{n,k}|\\
\le & {\mathscr F}_{0,\vv_n}^2|\vv_n|^2+{\mathscr F}_{1,\vv_n}^2|J_x \vv_n|^2
+2{\mathscr F}_{0,\vv_n}{\mathscr F}_{1,\vv_n}|\vv_n||J_x\vv_n|\\
\le& w_n^{\frac{2}{p}}({\mathscr F}_{0,\vv_n}^2+{\mathscr F}_{1,\vv_n}^2),
\end{align*}
where $c_p$ is the constant in Hypothesis \ref{sandro}(ii).
Putting everything together, we get
\begin{align*}
g\le \frac{1}{4c_p}|\vv_n|^2+ \bigg(r+(1-p)\mu_Q+\frac{1}{4(p-1)\mu_Q}{\mathscr Q}_1^2+c_p{\mathscr C}_1^2\bigg)|J_x \vv_n|^2.
\end{align*}

Using Hypothesis \ref{sandro}(ii) we conclude that $D_t w_n- {\mathcal A}w_n\le C_{1,p}w_n$ in $\Rd$
where $C_{1,p}$ is the maximum between $(4c_p)^{-1}$ and $K_p$. Hence, the function
$z_n(t,\cdot)= w_n(t,\cdot)- e^{C_{1,p}t}T_n^{\mathcal N}(t)(|\f|^2+|J\f|^2+\varepsilon)^{p/2}$ vanishes on $\{0\}\times B_n$,
satisfies the differential inequality
$D_t z_n-{\mathcal A}z_n-C_{1,p}z_n \le 0$ in $(0,+\infty)\times B_n$ and its normal derivative is
nonpositive on $(0,+\infty)\times \partial B_n$.
The classical maximum principle yields that $z_n\le 0$ in $(0,+\infty)\times B_n$,
whence, letting $\varepsilon \to 0^+$, estimate \eqref{stima_grad1_appr} follows at once.

{\em Step 2.} Now, we prove \eqref{der_est} with $k=1$ and $h=0$. Fix $t>0$. From \eqref{vector-scalar}, the previous step, the semigroup law and recalling that $(a+b)^{p/2}\le a^{p/2}+b^{p/2}$ for any $a,b\ge 0$, it follows that
\begin{align*}
|J_x\T(t)\f|^p&=|J_x \T(t-\sigma)\T(\sigma)\f|^p\\
&\le  e^{C_{1,p}(t-\sigma)}T(t-\sigma)[|\T(\sigma)\f|^p+|J_x\T(\sigma)\f|^p]\\
& \le   e^{C_{1,p} (t-\sigma)}\left[T(t)|\f|^p+T(t-\sigma)|J_x\T(\sigma)\f|^p\right]
\end{align*}
for any $\sigma \in (0,t)$. Formula \eqref{rep_int} and the H\"older inequality yield
\begin{align}
T(t-\sigma)&|J_x\T(\sigma)\f|^p
=T(t-\sigma)[|J_x\T(\sigma)\f|^p(|\T(\sigma)\f|^2\!+\!\delta)^{\frac{p(p-2)}{4}}
(|\T(\sigma)\f|^2\!+\!\delta)^{\frac{p(2-p)}{4}}]\notag\\
\le&\Big (T(t-\sigma)(|J_x\T(\sigma)\f|^2(|\T(\sigma)\f|^2+\delta)^{\frac{p-2}{2}})\Big)^{\frac{p}{2}}
\left(T(t-\sigma)(|\T(\sigma)\f|^2+\delta)^{\frac{p}{2}}\right)^{\frac{2-p}{2}}\notag\\
\le& \varepsilon^{\frac{2}{p}}\frac{p}{2}T(t-\sigma)\Big (|J_x \T(\sigma)\f|^2(|\T(\sigma)\f|^2+\delta)^{\frac{p-2}{2}}\Big )\notag\\
&+\left (1-\frac{p}{2}\right )\varepsilon^{\frac{2}{p-2}}T(t-\sigma)(|\T(\sigma)\f|^2+\delta)^{\frac{p}{2}}
\label{gianni}
\end{align}
for any $\varepsilon, \delta >0$, whence
\begin{align*}
e^{-C_{1,p}(t-\sigma)}|J_x\T(t)\f|^p \le &
T(t)|\f|^p+\left(1-\frac{p}{2}\right)\varepsilon^{\frac{2}{p-2}}T(t-\sigma)(|\T(\sigma)\f|^2+\delta)^{\frac{p}{2}}\nonumber
\\
&+\frac{p}{2}\,
\varepsilon^{\frac{2}{p}}T(t-\sigma)\left(|J_x\T(\sigma)\f|^2(|\T(\sigma)\f|^2+\delta)^{\frac{p-2}{2}}\right).
\end{align*}
Integrating the previous estimate with respect to $\sigma \in (0,t)$, we deduce
\begin{align}\label{pre}
|J_x\T(t)\f|^p \le \frac{C_{1,p}}{1\!-\!e^{-C_{1,p}t}}
\bigg\{&t T(t)|\f|^p\!+\!\bigg(1-\frac{p}{2}\bigg )\varepsilon^{\frac{2}{p-2}}
\int_0^tT(t-\sigma)(|\T(\sigma)\f|^2\!+\!\delta)^{\frac{p}{2}}d\sigma
\nonumber\\
&+\frac{p}{2}\varepsilon^{\frac{2}{p}}\int_0^tT(t-\sigma)
\left(|J_x \T(\sigma)\f|^2(|\T(\sigma)\f|^2\!+\!\delta)^{\frac{p-2}{2}}\right) d\sigma\bigg\}.
\end{align}
To prove the claim, we just need to show that there exists a positive constant $k_p$ such that
\begin{equation}\label{aim_int}
\int_0^tT(t-\sigma)
\left(|J_x \T(\sigma)\f|^2(|\T(\sigma)\f|^2+\delta)^{\frac{p-2}{2}}\right) d\sigma
 \le k_p T(t)(|\f|^2+\delta)^{\frac{p}{2}}
\end{equation}
for any $t>0$. Indeed, once \eqref{aim_int} is proved, replacing it into \eqref{pre},
letting $\delta \to 0^+$ (see \cite[Proposition 1.2.10]{newbook}), using again \eqref{vector-scalar}
and minimizing with respect to $\varepsilon>0$ (taking \eqref{maggiordomo} into account), we deduce that
\begin{align*}
|J_x\T(t)\f|^p \le &\frac{C_{1,p}}{1-e^{-C_{1,p}t}}\left\{\left[1+\bigg(1-\frac{p}{2}\bigg)
\varepsilon^{\frac{2}{p-2}}\right]t+ \frac{p}{2}\varepsilon^{\frac{2}{p}}k_p\right\} T(t)|\f|^p\\
 \le &\frac{C_{1,p}}{1-e^{-C_{1,p}t}}(t+k_p^{\frac{p}{2}}t^{1-\frac{p}{2}})T(t)|\f|^p,
\end{align*}
whence \eqref{der_est} with $k=1$ and $h=0$ follows.

We prove the above inequality with $\T(t)$ and $T(t)$ being replaced by $\T_n^{\mathcal N}(t)$ and $T^{\mathcal N}_n(t)$, respectively. Letting $n$ tend to $+\infty$, \eqref{aim_int} will follow at once.
We set
\begin{equation*}
\psi_n(\sigma,\cdot)=T_n^{\mathcal N}(t-\sigma)(|\vv_n(\sigma,\cdot)|^2+\delta)^{\frac{p}{2}}
=: T_n^{\mathcal N}(t-\sigma)v_n(\sigma, \cdot)
\end{equation*}
for any $\sigma \in [0,t]$ and $n \in \N$.
Since the normal derivative of the function $v_n(\sigma,\cdot)$ vanishes on $\partial B_n$ for any $\sigma\in (0,t)$,
$v_n(\sigma,\cdot)$ belongs to the domain of the generator of $T_n^{\mathcal N}(t)$ in $C_b(\overline B_n)$ for any $\sigma\in (0,t)$. Hence, $\psi_n$ is differentiable in $(0,t)$
\begin{align*}
\psi'_n&= T_n^{\mathcal N}(t-\cdot)(D_\sigma v_n-\mathcal{A} v_n)
\\
&=pT_n^{\mathcal N}(t-\cdot)\bigg[v_n^{1-\frac{2}{p}}\bigg (\langle C\vv_n,\vv_n\rangle-\sum_{i,j=1}^dq_{ij}\langle D_i\vv_n,D_j\vv_n\rangle\bigg )
\\
&\qquad\qquad\qquad\;\;\,+ (2-p)v_n^{1-\frac{4}{p}}\sum_{i,j=1}^d q_{ij}\langle \vv_n, D_i \vv_n\rangle \langle \vv_n, D_j \vv_n\rangle\bigg].
\end{align*}
Applying the same arguments as in Step 1, we deduce that
\begin{equation}\label{giampaolo}
\psi'_n(\sigma) \le p(1-p)\mu_0 T_n^{\mathcal N}(t-\sigma)[(v_n(\sigma,\cdot))^{1-\frac{2}{p}}|J_x\vv_n(\sigma,\cdot)|^2]
\end{equation}
for any $\sigma\in (0,t)$.
Thus \eqref{aim_int} follows with $k_p=[p(p-1)\mu_0]^{-1}$, simply integrating both sides of \eqref{giampaolo} with respect to $\sigma$ in $[h,t-h]$ and then letting $n\to+\infty$ and
$h\to 0^+$. The proof is so complete.
\end{proof}

Estimate \eqref{der_est} with $k=2$ is more involved and, as it has been already pointed out, it requires stronger assumptions on the coefficients of the operator $\bf{\A}$.

\begin{hyp}\label{sandro-2}
\begin{enumerate}[\rm (i)]
\item The coefficients $q_{ij}$, $b_i$ and the entries $c_{hk}$ of the matrix valued function $C$ belong to $C^{2+\alpha}_{\rm{loc}}(\Rd)$ for any $i,j=1,\ldots, d$ and $h,k=1, \ldots,m$;
\item there exist two positive constants $c_1,c_2$ such that for any $x\in \Rd$:
\begin{align}\label{dibala}
&|Q(x)|\le c_1(1+|x|^2)\mu_Q(x),\qquad\;\,|Q(x)x|\le c_1(1+|x|^2)\mu_Q(x),\\
&\langle b(x),x\rangle \le c_2(1+|x|^2)\mu_Q(x),
\label{pipita}
\end{align}

\item for any $p\in (1,+\infty)$, there exist positive constants $c_{jp}$ $(j=1,\ldots,6)$ such that
\begin{align*}
&K_{1p}:=\sup_{\Rd}\bigg (r-c_{3p}\mu_Q+c_{1p}{\mathscr C}_1^2+\frac{3}{2(p-1)\mu_Q}{\mathscr Q}_1^2+c_{2p}{\mathscr B}_2\bigg )<+\infty,\\
&K_{2p}:=\sup_{\Rd}\bigg (2r-c_{6p}\mu_Q+c_{4p}{\mathscr B}_2+c_{5p}{\mathscr C}_2^2+\frac{4}{(p-1)\mu_Q}{\mathscr Q}_1^2+{\mathscr Q}_2\bigg )<+\infty.
\end{align*}
where $r$ has been defined in Hypotheses $\ref{sandro}$.
%\item for any $p\in (1,+\infty)$, there exist two positive constants $c_{3p}$ and $c_{4p}$ such that
%\begin{equation}\label{mensa}
%r+c_{3p}\sum_{i=1}^d |D_i C|^2+\frac{1}{2(p-1)\mu_Q}\sum_{i=1}^d |D_i Q|^2\le c_{4p}\mu_Q
%\end{equation}
\end{enumerate}
\end{hyp}

\begin{thm}
\label{thm-4.4}
Under Hypotheses $\ref{sandro-2}$, estimate \eqref{der_est} holds true with $k=2$ and $\Gamma_{p,2,h}\sim \gamma'_{p,2,h}t^{-(2-h)p/2}$ as $t\to 0^+$ for some positive constant $\gamma'_{p,2,h}$ and $h=0,1,2$.
\end{thm}

\begin{proof}
We split the proof into three steps. In the first one we prove the claim with $h=2$.
Next we consider the case $h=1$ and, finally, $h=0$.

{\em Step 1.}
For $\alpha,\beta$, $\delta\in (0,1)$, $n\in\N$,
we define the function $z_n:[0,+\infty)\times\overline B_n\to\R$ by setting
$z_n(t,x)=(|\uu(t,x)|^2+\alpha\vartheta^2_n|J_x\uu(t,x)|^2+\beta\vartheta_n^4|D^2_x\uu(t,x)|^2+\delta)^{p/2}$, where
$\uu$ is the classical solution to the Cauchy problem \eqref{cantare}, $\vartheta_n(x)=\vartheta(|x|/n)$ for any $x\in\R^d$, $n\in\N$ and $\vartheta$ is a smooth function
such that $\chi_{(-1,1)}\le\vartheta\le\chi_{(-2,2)}$.

By Theorem \ref{teo-app}, $z_n\in C^{1,2}((0,+\infty)\times \Rd)\cap C_b([0,+\infty)\times \Rd)$ and
$D_t z_n-\mathcal A z_n= p z_n^{1-\frac{2}{p}}g$ in $(0,+\infty)\times \Rd$, where $g=\sum_{i=0}^6\psi_i+\frac{2-p}{4}z_n^{-\frac{2}{p}}|Q^{\frac{1}{2}}\zeta_n|^2$,
\begin{align*}
\psi_0=& \langle C\uu,\uu\rangle+\alpha \vartheta_n^2\sum_{i=1}^d \langle CD_i \uu,D_i\uu\rangle+\beta \vartheta_n^4\sum_{i,j=1}^d\langle C D_{ij}\uu,D_{ij}\uu\rangle
+\alpha \vartheta_n^2{\mathscr B}_{0,\uu},\\
\psi_1=&\alpha\vartheta_n^2\sum_{i,j,h=1}^d\sum_{k=1}^m D_hq_{ij}D_{ij} u_kD_hu_k+\alpha\vartheta_n^2\sum_{i=1}^d\sum_{k,s=1}^m D_i c_{ks}u_sD_i u_k,\\
\psi_2=&2\beta\vartheta_n^4\sum_{i,j,h,\ell=1}^d\sum_{k=1}^mD_hq_{ij}D_{h\ell}u_kD_{ij\ell}u_k
+\beta\vartheta_n^4\sum_{i,j,h,\ell=1}^d\sum_{k=1}^mD_{h\ell}q_{ij}D_{ij}u_kD_{h\ell}u_k,\\
\psi_3=&2\beta\vartheta_n^4\sum_{i,j=1}^d\sum_{k,s=1}^mD_ic_{ks}D_ju_sD_{ij}u_k+\beta\vartheta_n^4\sum_{i,j=1}^d\sum_{k,s=1}^mD_{ij}c_{ks}u_sD_{ij}u_k,\\
\psi_4=&2\beta\vartheta_n^4{\mathscr B}_{1,\uu}+\beta\vartheta_n^4\sum_{i,j=1}^d\sum_{k,s=1}^mD_{ij}b_sD_{ij}u_kD_su_k,\\
\psi_5=& -{\mathscr F}_{0,\uu}^2-\alpha\vartheta_n^2{\mathscr F}_{1,\uu}^2-\beta\vartheta_n^4{\mathscr F}_{2,\uu}^2-\alpha|Q^{\frac{1}{2}}\nabla \vartheta_n|^2|J_x\uu|^2-6\beta\vartheta_n^2|Q^{\frac{1}{2}}\nabla\vartheta_n|^2|D^2_x\uu|^2,\\
\psi_6 =&-4\alpha\vartheta_n\sum_{k=1}^m\langle Q\nabla \vartheta_n, D^2_x u_k\nabla_x u_k\rangle-2\vartheta_n\mathcal A\vartheta_n(2^{-1}\alpha|J_x\uu|^2+\beta\vartheta_n^2|D^2_x\uu|^2)\\
&-8\beta\vartheta_n^3\sum_{i,j=1}^d\langle Q\nabla\vartheta_n,\nabla D_{ij}u_k\rangle D_{ij}u_k
\end{align*}
and $\zeta_n=\nabla_x(|\uu|^2+\alpha\vartheta_n^2|J_x\uu|^2+\beta\vartheta_n^4|D^2_x\uu|^2) $.
Using the hypotheses, the Cauchy-Schwartz and Young's inequalities, as in the proof of Theorem \ref{thm-avvvooooccato}, we estimate the terms $\psi_i$, $i=0,\ldots,4$.
Clearly, $\psi_0\le \alpha r \vartheta_n^2|J_x\uu|^2$. Moreover,
\begin{align*}
\psi_1
\le& \alpha\vartheta_n^2\frac{1}{2c_{1p}}|\uu|^2+\alpha\vartheta_n^2\bigg (\frac{3}{2(p-1)\mu_Q}{\mathscr Q}_1^2+\frac{c_{1p}}{2}{\mathscr C}_1^2\bigg )|J_x \uu|^2
+\alpha\vartheta_n^2\frac{p-1}{6}\mu_Q|D_x^2\uu|^2,\\[2mm]
\psi_2\le & \frac{p-1}{4}\beta\vartheta_n^4\mu_Q|D_x^3\uu|^2+\beta\vartheta_n^4\bigg (\frac{4}{(p-1)\mu_Q}{\mathscr Q}_1^2+{\mathscr Q}_2\bigg )|D_x^2 \uu|^2,\\[2mm]
\psi_3\le & \frac{1}{4c_{5p}}\beta\vartheta_n^4|\uu|^2+\beta c_{1p}\vartheta_n^4{\mathscr C}_1^2|J_x \uu|^2
+\beta\vartheta_n^4\bigg (\frac{1}{c_{1p}}+c_{5p}{\mathscr C}_2^2\bigg )|D_x^2\uu|^2,\\[2mm]
\psi_4\le& \beta\vartheta_n^4(2r+c_{4p}{\mathscr B}_2)|D_x^2\uu|^2+\frac{\beta}{4c_{4p}}\vartheta_n^4{\mathscr B}_2|J_x\uu|^2.
\end{align*}
Now, we observe that
\begin{align*}
\frac{1}{4}|Q^{\frac{1}{2}}\zeta_n|^2\!\le
&[{\mathscr F}_{0,\uu}|\uu|+\alpha\vartheta_n(\vartheta_n{\mathscr F}_{1,\uu}+|Q^{\frac{1}{2}}\nabla\vartheta_n||J_x\uu|)|J_x\uu|\\
&\;+\beta\vartheta_n^2(\vartheta_n^2{\mathscr F}_{2,\uu}+2\vartheta_n|Q^{\frac{1}{2}}\nabla\vartheta_n||D^2_x\uu|)|D^2_x\uu|]^2\\
\le &(|\uu|^2+\alpha\vartheta_n^2|J_x\uu|^2+\beta\vartheta_n^4|D^2_x \uu|^2)\\
&\;\times\![{\mathscr F}_{0,\uu}^2\!+\!\alpha(\vartheta_n{\mathscr F}_{1,\uu}\!+\!|J_x\uu||Q^{\frac{1}{2}}\!\nabla\vartheta_n|)^2\!+\!\beta(\vartheta_n^2\!{\mathscr F}_{2,\uu}\!+\!2\vartheta_n|D^2_x\uu||Q^{\frac{1}{2}}\!\nabla\vartheta_n|)^2]\\
\le &  z_n^{\frac{2}{p}}[{\mathscr F}_{0,\uu}^2+(1+\varepsilon)\alpha\vartheta_n^2{\mathscr F}_{1,\uu}^2
+(1+\varepsilon)\beta\vartheta_n^4{\mathscr F}_{2,\uu}^2\\
&\quad\;+\varepsilon^{-1}(1+\varepsilon)|Q^{\frac{1}{2}}\nabla\vartheta_n|^2(\alpha|J_x\uu|^2+4\beta\vartheta_n^2|D^2_x\uu|^2)],
\end{align*}
where we used the estimate $(a+b)^2\le (1+\varepsilon)(a^2+\varepsilon^{-1}b^2)$ which holds true for any $a,b,\varepsilon>0$. Consequently,
taking $\varepsilon=(p-1)[2(2-p)]^{-1}$, we get
\begin{align*}
\psi_5+\frac{2-p}{4}z_n^{-\frac{2}{p}}|Q^{\frac{1}{2}}\zeta_n|^2
\le &\frac{1-p}{2}(\mu_Q|J_x\uu|^2+\alpha\vartheta_n^2{\mathscr F}_{1,\uu}^2+\beta\vartheta_n^4{\mathscr F}_{2,\uu}^2)\\
&+\frac{(p^2-6p+7)^+}{p-1}|Q^{\frac{1}{2}}\nabla\vartheta_n|^2(\alpha|J_x\uu|^2+4\beta\vartheta_n^2|D^2_x\uu|^2).
\end{align*}
Since
$D_i\vartheta_n(x)=x_i(|x|n)^{-1}\vartheta'(n^{-1}|x|)$ for any $x\in\Rd$ and $i=1,\ldots,d$, applying \eqref{dibala} and Young's inequality we conclude that
$|Q^{\frac{1}{2}}\nabla\vartheta_n|^2\leq M_1\mu_Q$ in $\Rd$ for some positive constant $M_1$.
Hence, we obtain
\begin{align*}
\psi_5+\frac{2-p}{4}z_n^{-\frac{2}{p}}|Q^{\frac{1}{2}}\zeta_n|^2
\le &\frac{1-p}{2}(\mu_Q|J_x\uu|^2+\alpha\vartheta_n^2{\mathscr F}_{1,\uu}^2+\beta\vartheta_n^4{\mathscr F}_{2,\uu}^2)\\
&+M_1\frac{(p^2-6p+7)^+}{p-1}\mu_Q(\alpha|J_x\uu|^2+4\beta\vartheta_n^2|D^2_x\uu|^2).
\end{align*}
It remains to estimate $\psi_6$. As above, taking the choice of $\vartheta$, \eqref{dibala} and \eqref{pipita} into account, we deduce that
$-{\mathcal A}\vartheta_n\le M_2\mu_Q$ for some positive constant $M_2$.
It thus follows that
\begin{align*}
\psi_6\le & 4\alpha\vartheta_n|Q^{\frac{1}{2}}\nabla\vartheta_n|{\mathscr F}_{1,\uu}|J_x\uu|+2M_2\vartheta_n\mu_Q\bigg(\frac{\alpha}{2}|J_x\uu|^2+\beta\vartheta_n^2|D^2_x\uu|^2\bigg)\\
&+8\beta\vartheta_n^3|Q^{\frac{1}{2}}\nabla\vartheta_n||D^2_x\uu|{\mathscr F}_{2,\uu}\\
\le& \alpha\vartheta_n^2\frac{p-1}{4}{\mathscr F}_{1,\uu}^2+
\alpha \bigg (\frac{16}{p-1}M_1+\vartheta_nM_2\bigg )\mu_Q|J_x\uu|^2\\
&+\beta \vartheta_n^2\bigg (2M_2\vartheta_n+\frac{64}{p-1}M_1\bigg)\mu_Q|D^2_x\uu|^2
+\frac{p-1}{4}\beta\vartheta_n^4{\mathscr F}_{2,\uu}^2.
\end{align*}
Summing up we deduce that
\begin{align*}
g\le& K_{3p}|\uu|^2+\bigg\{\vartheta_n^2\bigg[\alpha\bigg (r+\frac{3}{2(p-1)\mu_Q}{\mathscr Q}_1^2+\frac{c_{1p}}{2}{\mathscr C}_1^2\bigg )+\beta c_{1p}{\mathscr C}_1^2
+\frac{\beta}{4c_{4p}}{\mathscr B}_2\bigg ]\\
&\qquad+\bigg [\frac{1-p}{2}+\alpha \bigg (M_1\frac{(p^2-6p+7)^+}{p-1}+\frac{16}{p-1}M_1+M_2\bigg )\bigg ]\mu_Q\bigg\}|J_x\uu|^2\\
&+\vartheta_n^2\bigg [\bigg (\alpha\frac{1-p}{12}+2\beta M_2+\beta\frac{64}{p-1}M_1+4\beta M_1\frac{(p^2-6p+7)^+}{p-1}\bigg )\mu_Q\\
&\qquad\;\,+\beta\vartheta_n^2\bigg (2r+\frac{1}{c_{1p}}+c_{5p}{\mathscr C}_2^2+c_{4p}{\mathscr B}_2
+\frac{4}{(p-1)\mu_Q}{\mathscr Q}_1^2+{\mathscr Q}_2\bigg )\bigg ]|D^2_x\uu|^2,
\end{align*}
where
$K_{3p}=K_{3p}(\alpha,\beta)=2c_{1p}^{-1}\alpha+(4c_{5p})^{-1}\beta$.
We choose $\alpha=\alpha_p$ sufficiently small such that
\begin{eqnarray*}
\alpha\bigg (M_1\frac{(p^2-6p+7)^+}{p-1}+\frac{16}{p-1}M_1+M_2\bigg )\le \frac{p-1}{4}
\end{eqnarray*}
and, then, $\beta\in (0,\alpha/2)$ such that
\begin{eqnarray*}
\beta\bigg (2M_2+\frac{64}{p-1}M_1+4M_1\frac{(p^2-6p+7)^+}{p-1}\bigg )\le \alpha\frac{p-1}{24}.
\end{eqnarray*}
With these choices of $\alpha$ and $\beta$, we conclude that
\begin{align*}
g\le& K_{3p}|\uu|^2+\vartheta_n^2\alpha\bigg(\frac{1-p}{4\alpha}\mu_Q+r+\frac{3}{2(p-1)\mu_Q}{\mathscr Q}_1^2+c_{1p}{\mathscr C}_1^2+\frac{\beta}{4\alpha c_{4p}}{\mathscr B}_2\bigg )|J_x\uu|^2\\
&+\beta\vartheta_n^4\bigg (\alpha\frac{1-p}{24\beta}\mu_Q+2r+\frac{1}{c_{1p}}+c_{5p}{\mathscr C}_2^2+c_{4p}{\mathscr B}_2
+\frac{4}{(p-1)\mu_Q}{\mathscr Q}_1^2+{\mathscr Q}_2\bigg )|D^2_x\uu|^2.
\end{align*}

Taking $\alpha$ and $\beta$ smaller if needed, we can assume that $g \le K_{4p}z_n^{2/p}$. Now, arguing as in the last part of the proof of Theorem \ref{thm-avvvooooccato} and letting $n$ tend to $+\infty$, estimate
\eqref{der_est} follows in this case.

{\em Step 2.}
Now, we consider the case $h=1$. Fix $t>0$. From Step 1 we get
\begin{align*}
&(|J_x\uu(t,\cdot)|^2+|D^2_x \uu(t,\cdot)|^2)^{\frac{p}{2}}\\
=&(|J_x \T(t-\sigma)\uu(\sigma,\cdot)|^2+|D^2_x\T(t-\sigma)\uu(\sigma,\cdot)|^2)^{\frac{p}{2}}\\
\le &\Gamma_{p,2,2}(t-\sigma)T(t-\sigma)(|\uu(\sigma,\cdot)|^2+|J_x\uu(\sigma,\cdot)|^2+|D^2_x\uu(\sigma,\cdot)|^2)^{\frac{p}{2}}\\
\le &\Gamma_{p,2,2}(t-\sigma)T(t)|\f|^p+\Gamma_{p,2,2}(t-\sigma)T(t-\sigma)(|J_x\uu(\sigma,\cdot)|^2+|D^2_x \uu(\sigma,\cdot)|^2)^{\frac{p}{2}}
\end{align*}
for any $\sigma \in (0,t)$, where we used the estimate $(a+b)^{p/2}\le a^{p/2}+b^{p/2}$ which holds true for any $a,b\ge 0$. For any $\delta>0$ we set $u(\sigma,\cdot)=(|\uu(\sigma,\cdot)|^2+|J_x\uu(\sigma,\cdot)|^2+\delta)^{\frac{p}{2}}$, $\sigma \in (0,t)$.
Arguing as in \eqref{gianni}, we can estimate
\begin{align*}
&T(t-\sigma)(|J_x\uu(\sigma,\cdot)|^2+|D^2_x \uu(\sigma,\cdot)|^2)^{\frac{p}{2}}\\
\le &\varepsilon^{\frac{2}{p}}\frac{p}{2}T(t\!-\!\sigma)[(|J_x\uu(\sigma,\cdot)|^2\!+\!|D^2_x \uu(\sigma,\cdot)|^2)(u(\sigma,\cdot))^{1-\frac{2}{p}}]
\!+\!\left (1-\frac{p}{2}\right )\varepsilon^{\frac{2}{p-2}}T(t\!-\!\sigma)u(\sigma,\cdot)
\end{align*}
for any $\sigma\in (0,t)$ and $\varepsilon>0$. Since $\Gamma_{p,2,2}(t-\sigma)=e^{C_p(t-\sigma)}$ for some positive constant $C_p$, combining the previous two estimates
 we deduce
\begin{align}\label{wferrara}
&e^{-C_{p}(t-\sigma)}(|J_x\uu(t,\cdot)|^2+|D^2_x \uu(t,\cdot)|^2)^{\frac{p}{2}}\nonumber\\
\le &T(t)|\f|^p+\left(1-\frac{p}{2}\right)\varepsilon^{\frac{2}{p-2}}T(t-\sigma)u(\sigma,\cdot)\nonumber\\
&+\frac{p}{2}\,
\varepsilon^{\frac{2}{p}}T(t-\sigma)\left[\left(|J_x\uu(\sigma,\cdot)|^2+|D^2_x \uu(\sigma,\cdot)|^2\right)u(\sigma,\cdot)^{1-\frac{2}{p}}\right].
\end{align}
Notice that the proof of Theorem \ref{thm-avvvooooccato} shows that
\begin{align*}
T(t-\sigma)u(\sigma,\cdot)\le &
T(t-\sigma)(|\uu(\sigma,\cdot)|^p+(|J_x\uu(\sigma,\cdot)|^2+\delta)^{\frac{p}{2}})\\
\le &T(t-\sigma)[T(\sigma)|\f|^p+\Gamma_{2,1,1}(\sigma)T(\sigma)(|\f|^2+|J\f|^2+\delta)^{\frac{p}{2}}]\\
\le & e^{c_p\sigma}T(t)(|\f|^2+|J\f|^2+\delta)^{\frac{p}{2}}+T(t)|\f|^p
\end{align*}
for any $\sigma \in (0,t)$ and some positive constant $c_p$. Consequently, integrating \eqref{wferrara} with respect to $\sigma \in (0,t)$ we deduce that
\begin{align*}
|D^2_x \uu(t,\cdot)|^p\le\frac{C_p}{1\!-\!e^{-C_pt}}\!\bigg \{&\frac{p}{2}\varepsilon^{\frac{2}{p}}\!\int_0^t\!T(t-\sigma)[(|J_x\uu(\sigma,\cdot)|^2\!+\!|D^2_x \uu(\sigma,\cdot)|^2)(u(\sigma,\cdot))^{1\!-\!\frac{2}{p}}]d\sigma\\
&+\left [1+\left (1-\frac{p}{2}\right )\varepsilon^{\frac{2}{p-2}}\right ]tT(t)|\f|^p\\
&+\left (1-\frac{p}{2}\right )\varepsilon^{\frac{2}{p-2}} \frac{e^{c_p t}-1}{c_p}T(t)(|\f|^2+|J\f|^2+\delta)^{\frac{p}{2}}\bigg \}.
\end{align*}
Now we claim that there exists a positive constant $K_p$ such that
\begin{align}\label{cinema}
&\int_0^tT(t-\sigma)[(|J_x\uu(\sigma,\cdot)|^2+|D^2_x \uu(\sigma,\cdot)|^2)(u(\sigma,\cdot))^{1-\frac{2}{p}}]d\sigma\nonumber\\
\le &K_p\bigg (T(t)(|\f|^2+|J\f|^2+\delta)^{\frac{p}{2}}+ \int_0^t T(t-\sigma)u(\sigma,\cdot)d\sigma\bigg ).
\end{align}
Once \eqref{cinema} is proved, using again \eqref{wferrara} we deduce
\begin{align*}
|D^2_x \uu(t,\cdot)|^p
%&\le\frac{C_p}{1-e^{-C_pt}}\bigg [\bigg (2+\left (1-\frac{p}{2}\right )\varepsilon^{\frac{2}{p-2}}\bigg )t T(t)|\f|^p+\gamma_{\varepsilon,p}(t)T(t)(|\f|^2+|J\f|^2+\delta)^{p/2}\Big)\\
&\le \frac{C_p}{1-e^{-C_pt}}[t+\gamma_{\varepsilon,p}(t)]T(t)(|\f|^2+|J\f|^2+\delta)^{\frac{p}{2}},
\end{align*}
where
\begin{eqnarray*}
\gamma_{\varepsilon,p}(t)=\left [\left (1-\frac{p}{2}\right )\varepsilon^{\frac{2}{p-2}}\!+\!K_p\frac{p}{2}\varepsilon^{\frac{2}{p}}\right ] \frac{e^{c_p t}-1}{c_p}+\frac{p}{2}\varepsilon^{\frac{2}{p}}K_p\!+\!\left (1-\frac{p}{2}\right )\varepsilon^{\frac{2}{p-2}}t+\frac{p}{2}\varepsilon^{\frac{2}{p}}K_pt.
\end{eqnarray*}
Letting $\delta \to 0^+$ and minimizing on $\varepsilon>0$ we obtain \eqref{der_est} with $k=2$, $h=1$ and
\begin{eqnarray*}
\Gamma_{p,2,1}(t)=\frac{C_p}{1-e^{-C_pt}}\bigg [t+K_p^{\frac{p}{2}}\left (\frac{e^{c_pt}-1}{c_p}+t\right )^{1-\frac{p}{2}}\left (\frac{e^{c_pt}-1}{c_p}+1+t\right )^{\frac{p}{2}}\bigg ].
\end{eqnarray*}

To conclude, we prove \eqref{cinema}. To this aim we introduce the same sequence of cut-off functions as in Step 1.
For any $\alpha>0$, $t>0$, $\delta\in (0,1)$, $n\in\N$,  and $f\in C^2_b(\R^d)$
we define the function $\psi_n:[0,t]\to C(\overline B_n)$ by setting
$\psi_n(\sigma)=T_n^{\mathcal D}(t-\sigma)(u_{n}(\sigma,\cdot)-\delta^{p/2})$ for any $\sigma\in [0,t]$, where
$u_n(\sigma,\cdot)=(|\uu_n(\sigma,\cdot)|^2+\alpha\vartheta^2_n|J_x\uu_n(\sigma,\cdot)|^2+\delta)^{p/2}$,
$(\uu_n)$ is the sequence of  solutions to the Cauchy-Dirichlet problems \eqref{prob_approx_Dir} and $T_n^{\mathcal D}(t)$ 
is the positive semigroup associated to the realization of $\mathcal{A}$ in $C_b(\overline B_n)$ with homogeneous Dirichlet
boundary conditions.

Since $u_n(\sigma,\cdot)-\delta^{p/2}$ vanishes on $\partial B_n$ for any $\delta>0$, taking \cite[Theorem 2.3(ix)]{acqui} into account, we can show that the function $\psi_n$
is differentiable in $(0,t)$ and
\begin{align*}
\psi'_n(\sigma)=p T_n^{\mathcal D}(t-\sigma)(u_n^{1-\frac{2}{p}}g),
\end{align*}
where $g$ is as in Step 1 with $\beta=0$ and $\uu$ being replaced by $\uu_n$. Hence, we can estimate
\begin{align*}
g\le& \frac{\alpha}{2c_{1p}}|\uu_n|^2\!+\!\bigg\{\alpha\vartheta_n^2\bigg (r+\frac{3}{2(p-1)\mu_Q}{\mathscr Q}_1^2+\frac{c_{1p}}{2}{\mathscr C}_1^2\bigg )\\
&\qquad\qquad\quad\;\,+\bigg [\frac{1-p}{2}\!+\!\alpha \bigg (M_1\frac{(p^2-6p+7)^+}{p-1}\!+\!\frac{16}{p-1}M_2^2+\frac{M_2}{2}\bigg )\bigg ]\mu_Q\bigg\}|J_x\uu_n|^2\\
&+\alpha\frac{1-p}{12}\mu_Q\vartheta_n^2|D^2_x\uu_n|^2.
\end{align*}
The coefficient in front of $|D^2_x\uu_n|^2$ is clearly negative and choosing properly $\alpha$ we can make negative also the coefficient in front of $|J_x\uu_n|^2$.
In this way we conclude that
$pu_n^{1-2/p}g\le -K_0p(|J_x\uu_n|^2+|D^2_x\uu_n|^2)\vartheta_n^2u_n^{1-2/p} +K_1pu_n$
for some positive constants $K_0$ and $K_1$. Thus,
\begin{equation*}
\psi'_n\le -K_0pT_n^\mathcal{D}(t-\cdot)[(|J_x\uu_n|^2+|D^2_x\uu_n|^2)\vartheta_n^2u_n^{1-\frac{2}{p}}]+K_1pT_n^\mathcal{D}(t-\cdot)u_n
\end{equation*}
in $(0,t)$,
which, we integrate with respect to $\sigma \in (\varepsilon, t-\varepsilon)$, $\varepsilon>0$. Letting first $n$ tend to $+\infty$ and then $\varepsilon$ tend to $0^+$, \eqref{cinema} follows.

{\em Step 3.} Estimate \eqref{der_est} with $k=2$ and $h=0$ can be obtained by the previous step, the semigroup law, \eqref{vector-scalar} and Theorem \ref{thm-avvvooooccato}.
Indeed, we have
\begin{align*}
|D^2_x\T(t)\f|^p=&|D^2_x\T(t/2)\T(t/2)\f|^p\le \Gamma_{p,2,1}(t/2)T(t/2)(|\T(t/2)\f|^2+|J_x\T(t/2)\f|^2)^{\frac{p}{2}}\\
\le &\Gamma_{p,2,1}(t/2)T(t/2)(|\T(t/2)\f|^p+|J_x\T(t/2)\f|^p)\\
\le &\Gamma_{p,2,1}(t/2)T(t)|\f|^p+\Gamma_{p,2,1}(t/2)\Gamma_{p,1,0}(t/2)T(t)|\f|^p
\end{align*}
for any $t>0$,
whence the claim follows with $\Gamma_{p,2,0}(t)=\Gamma_{p,2,1}(t/2)(1+\Gamma_{p,1,0}(t/2))$.
\end{proof}

\begin{example}{\rm
Let $\A$ be as in Example \ref{example-1} and assume further that, for any $k,s=1,\ldots,m$, the function $c_{ks}$ belongs to $C^{1+\alpha}_{\rm loc}(\Rd)$ and $|\nabla c_{ks}(x)|=O(|x|^{\tau})$, as $|x|\to +\infty$, with $0<\tau<\beta\vee \gamma$, then also Hypotheses \ref{sandro} hold true and Theorem \ref{thm-avvvooooccato} can be applied.
Indeed, since $\mu_{Q}(x)=(1+|x|^2)^{\gamma} \mu_0$ for any $x\in\Rd$, where $\mu_0>0$ is the minimum eigenvalue of the constant matrix $Q^0$, and the function $r$, which bounds form above the quadratic form associated to the Jacobian matrix of ${\bf b}$, is given by $r(x)=-b_0(1+|x|^2)^{\beta}$ for any $x \in \Rd$, the sum of the first two terms in the definition of $K_p$, which is the ``good'' part of $K_p$ (see Hypotheses \ref{sandro}), behaves like $|x|^{2(\beta\vee \gamma)}$ as $|x|\to +\infty$. Now it is immediate to check that $\mu_Q^{-1}{\mathscr Q}_1^2=O(|x|^{2\gamma-2})$ and ${\mathscr C}_1^2=O(x^{2\tau})$ as $|x|\to +\infty$, where we use Landau's formalism. Hence the supremum in Hypothesis \ref{sandro}(ii) is finite and estimate \eqref{der_est} with $k=1$, $h=0,1$ holds true.

Without much effort one can realize that the functions ${\mathscr B}_2$ and ${\mathscr Q}_2$ grow at infinity as $|x|^{2\beta-1}$ and $|x|^{2\gamma-2}$, respectively. Thus, if $c_{ij}\in C^{2+\alpha}_{\rm loc}(\Rd)$ and $|D^2c_{ks}(x)|^2=O(|x|^{\tau})$ as $|x|\to +\infty$, for any $k,s=1,\ldots,m$, then Hypotheses \ref{sandro-2} are satisfied too and Theorem \ref{thm-4.4} can be applied.} \end{example}

Starting from estimate \eqref{der_est}, it is routine to prove the following partial characterization of $D({\bf A}_p)$.

\begin{coro}
\label{coro-3.12}
Under Hypotheses $\ref{sandro}$, for any $t>0$ and $p\in (1,+\infty)$ the operator $\T(t)$ is bounded from $L^p_{\bm\mu}(\R^d;\R^m)$ into $W^{1,p}_{\bm\mu}(\Rd;\Rm)$ and
there exist two positive constants $N_{0,p}$ and $\omega_{1,p}$ such that
\begin{align}
\|\T(t)\f\|_{1,p,\bm\mu}\le N_{0,p}\max\{t^{-\frac{1}{2}},1\}\|\f\|_{p,\bm\mu},\qquad\;\,t>0,\;\,\f\in L^p_{\bm\mu}(\Rd;\Rm);
\label{none}\\[1mm]
\|\T(t)\f\|_{1,p,\bm\mu}\le N_{0,p}e^{\omega_{1,p}t}\|\f\|_{1,p,\bm\mu},\qquad\;\,t>0,\;\,\f\in W^{1,p}_{\bm\mu}(\Rd;\Rm).
\label{none-1}
\end{align}
Moreover, $D({\bf A}_p)$ is continuously embedded into $W^{1,p}_{\bm\mu}(\Rd;\R^m)$ for any $p\in (1,+\infty)$.

If also Hypotheses $\ref{sandro-2}$ are satisfied, then each operator $\T(t)$ is bounded from $L^p_{\bm\mu}(\R^d;\R^m)$ into $W^{2,p}_{\bm\mu}(\Rd;\Rm)$,
for any $p\in (1,+\infty)$, and
there exist two positive constants $N_{2,p}$ and $\omega_{2,p}$ such that
\begin{align}
\|\T(t)\f\|_{2,p,\bm\mu}\le N_{2,p}\max\{t^{-1},1\}\|\f\|_{p,\bm\mu},\qquad\;\,t>0,\;\,\f\in L^p_{\bm\mu}(\Rd;\Rm);
\label{larai}
\\[1mm]
\|\T(t)\f\|_{2,p,\bm\mu}\le N_{2,p}t^{-\frac{2-j}{2}t}e^{\omega_{2,p}t}\|\f\|_{j,p,\bm\mu},\qquad\;\,t>0,\;\,\f\in W^{j,p}_{\bm\mu}(\Rd;\Rm),\;\,j=1,2.
\label{larai-1}
\end{align}
\end{coro}

\begin{proof}
To begin with, we prove estimates \eqref{none} and \eqref{none-1}; the proofs of \eqref{larai} and \eqref{larai-1} are completely similar. Fix $p\in (1,+\infty)$, $\f\in C^{\infty}_c(\Rd;\Rm)$. Integrating
\eqref{der_est} with respect to the measure $\mu$ we obtain that
\begin{align*}
\int_{\Rd}|J_x\T(t)\f|^pd\mu\le &\Gamma_{p,1,h}(t)\int_{\Rd}T(t)\bigg (\sum_{k=0}^j|D^k\f|^2\bigg )^{\frac{p}{2}}d\mu\\
=&\Gamma_{p,1,h}(t)\int_{\Rd}\bigg (\sum_{k=0}^h|D^k\f|^2\bigg )^{\frac{p}{2}}d\mu
\end{align*}
for any $t>0$. Since $\mu_i=\widetilde c_i\mu$ for some positive constant $\widetilde c_i$ and any $i=1,\ldots,m$ (see Theorem \ref{exi_inv_meas}),
from the previous estimate, we conclude that
\begin{eqnarray*}
\|\T(t)\f\|_{1,p,\bm\mu}\le \widetilde N_{0,p}e^{\omega_{1,p}t}t^{-\frac{1-j}{2}t}\|\f\|_{j,p,\bm\mu},\qquad\;\,t>0,\;\,j=0,1,
\end{eqnarray*}
for some positive constants $\widetilde N_{0,p}$ and $\omega_{1,p}$. Estimate \eqref{none-1} follows by a density argument, taking Remark \ref{rm_dens} into account.
If $j=0$, then we can remove the exponential term from the right-hand side of the previous estimate.
Indeed, for $t>1$, using \eqref{festa} we can estimate
\begin{align*}
\|\T(t)\f\|_{1,p,\bm\mu}=&\|\T(1)\T(t-1)\f\|_{1,p,\bm\mu}\le \widetilde N_{0,p}e^{\omega_{1,p}}\|\T(t-1)\f\|_{p,\bm\mu}\\
\le &2^{\frac{p-1}{p}}\widetilde N_{0,p}e^{\omega_{1,p}}\|\f\|_{p,\bm\mu}
\end{align*}
and \eqref{none} follows, again by a density argument.

Let us now complete the proof, by showing that $D({\bf A}_p)\hookrightarrow W^{1,p}_{\bm\mu}(\Rd;\Rm)$.
We fix $p\in (1,+\infty)$ and observe that, in view of \eqref{none}, the resolvent operator $R(\lambda,{\bf A}_p)$ is defined for any $\lambda>0$ and
\begin{eqnarray*}
R(\lambda,{\bf A}_p)\f=\int_0^{+\infty}e^{-\lambda t}\T(t)\f dt,\qquad\;\,\f\in L^p_{\bm\mu}(\R^d;\R^m).
\end{eqnarray*}

Fix $\lambda>0$, $\uu\in D({\bf A}_p)$ and let $\f\in L^p_{\bm\mu}(\R^d;\R^m)$ be such that $\uu=R(\lambda,{\bf A}_p)\f$.
Using \eqref{der_est} with $h=0$ and $k=1$, we can estimate
\begin{align*}
\|J_x\uu\|_{p,\bm\mu}\le &c_{1,p}\|\f\|_{p,\bm\mu}\int_0^{+\infty}t^{-\frac{1}{2}}e^{-\lambda t}dt
\le c_{2,p}\lambda^{-\frac{1}{2}}\|\f\|_{p,\bm\mu}\\
\le &c_{2,p}(\lambda^{\frac{1}{2}}\|\uu\|_{p,\bm\mu}+\lambda^{-\frac{1}{2}}\|{\bf A}_p\uu\|_{p,\bm\mu})
\end{align*}
for some positive constants $c_{1,p}$ and $c_{2,p}$, independent of $\uu$. The previous chain of inequalities
shows that $D({\bf A}_p)$ is compactly embedded into $W^{1,p}_{\bm\mu}(\Rd;\R^m)$. Moreover, minimizing with respect to $\lambda>0$
we also conclude that
\begin{eqnarray*}
\|J_x\uu\|_{p,\bm\mu}\le c_{3,p}\|\uu\|_{p,\bm\mu}^{\frac{1}{2}}\|{\bf A}_p\uu\|_{p,\bm\mu}^{\frac{1}{2}},
\end{eqnarray*}
the constant $c_{3,p}$ being independent of $\uu$.
\end{proof}

\section{Asymptotic behaviour}
\label{sect:4}
To study the asymptotic behaviour of $\T(t)\f$ as $t \to +\infty$ we need some additional hypotheses.

\begin{hyp0}
\label{hyp-nonsichiama}
The coefficients of the operator $\A$ belong to $C^{1+\alpha}_{\rm loc}(\Rd)$. Moreover, there exists a positive constant $c$ such that
\begin{align}
(i)~|q_{ij}(x)|\le c(1+|x|^2)\varphi(x),\qquad\;\, (ii)~\langle {\bf b}(x),x\rangle \le c(1+|x|^2)\varphi(x),
\label{nonsichiama}
\end{align}
for any $x\in\Rd$ and $i,j=1,\ldots,m$.
\end{hyp0}

\begin{rmk}
{\rm We stress that, in general, Hypothesis \ref{hyp-nonsichiama} is not implied by Hypothesis \ref{hyp-base}(iv).
Consider for instance the one-dimensional operator $\mathcal A=qD_x^2+bD_x$, where
$q(x)=(1+x^2)e^{x^4}$ and $b(x)=-3x(1+x^2)e^{x^4}$ for any $x\in\R$.
It is easy to check that the function $x\mapsto x^2+1$ satisfies Hypothesis \ref{hyp-base}(iv).

We claim that no function $\varphi$ satisfying both Hypothesis \ref{hyp-base}(iv) and Hypothesis \ref{hyp-nonsichiama} exists.
We argue by contradiction and assume that such a function exists.
To begin with, we observe that \eqref{nonsichiama}(i) implies that
\begin{equation}
\varphi(x)\geq ce^{x^4},\qquad\;\,x\in\R,
\label{impossibile}
\end{equation}
for some positive constant $c$.
From this condition we can easily deduce that there exists an increasing sequence $(x_n)$, which blows up as $n\to +\infty$,
such that
\begin{equation}
\varphi'(x_n)> cx_n^3e^{x_n^4}-1,\qquad\;\, n\in\N.
\label{impossibile-1}
\end{equation}
Indeed, if this were not the case, there would exist
$M_1>0$ such that $\varphi'(x)\leq cx^3e^{x^4}-1$ for any $x\ge M_1$. This inequality, integrated between $M_1$ and $x$, gives
$\varphi(x)\le ce^{x^4}/4+K$ for some positive constant $K$, which, clearly, contradicts \eqref{impossibile}.

Since we are assuming that $\varphi$ satisfies Hypothesis \ref{hyp-base}(iv), we can determine a positive constant $M_2$ such that
$q\varphi''+b\varphi'\leq 0$ in $[M_2,+\infty)$, from which we deduce that
$\varphi''(x)\le 3x \varphi'(x)$ for any $x\ge M_2$. Let $n_0$ be the smallest integer such that $x_{n_0}\ge M_2$.
Then, from the previous differential inequality we can infer that
$\varphi'(x)\le e^{3x^2/2}\varphi'(x_{n_0})$ for any $x\ge x_{n_0}$.
This estimate combined with \eqref{impossibile-1} leads us to a contradiction.}
\end{rmk}

The following result plays a crucial role in the study of the asymptotic behaviour of the function $\T(t)\f$ as $t\to +\infty$.

\begin{prop}
\label{metropolitana}
Under Hypothesis $\ref{hyp-nonsichiama}$, for any $\f\in C^{3+\alpha}_c(\Rd;\Rm)$
the $L^2_{\mu}(\Rd)$-norm of $|J_x\T(t)\f|$ vanishes as $t$ tends to $+\infty$.
\end{prop}

\begin{proof}
To begin with, we recall that
\begin{equation}
\int_{\Rd}{\mathcal A}\psi d\mu=0,
\label{infinitesima-invariant}
\end{equation}
for any bounded function $\psi\in D_{\max}(\mathcal A)=\{u\in C_b(\Rd)\cap\bigcap_{p<+\infty}W^{2,p}_{\rm loc}(\Rd): {\mathcal A}u\in C_b(\Rd)\}$ (see \cite[Proposition 9.1.1]{newbook}).
In particular, if $\psi\in C^2_c(\Rd)$, then $\psi^2\in C^2_c(\Rd)$ and writing \eqref{infinitesima-invariant} with $\psi$ being replaced by
$\psi^2$, we easily conclude that
\begin{equation}
\int_{\Rd}\psi{\mathcal A}\psi d\mu=-\int_{\Rd}|Q^{\frac{1}{2}}\nabla\psi|^2 d\mu.
\label{progetto}
\end{equation}

We now introduce a decreasing function $\vartheta\in C^2(\R)$ such that $\chi_{(-\infty,1]}\le\vartheta\le\chi_{(-\infty,2]}$ and, for any $n\in\N$ and $x\in\Rd$, we set
$\vartheta_n(x)=\vartheta(n^{-1}|x|)$.

As it is immediately seen, for any $\f\in C_b(\Rd;\Rm)$ and $t>0$, the function $\vartheta_n\T(t)\f$ belongs to $C^2_c(\Rd;\Rm)$, so that, by
\eqref{progetto}, it follows that
\begin{equation}
\int_{\Rd}(\vartheta_n\T(t)\f)_j{\mathcal A}(\vartheta_n\T(t)\f)_j d\mu=-\int_{\Rd}|Q^{\frac{1}{2}}\nabla_x (\vartheta_n\T(t)\f)_j|^2 d\mu,
\label{progetto-1}
\end{equation}
for any $j=1,\ldots,m$ and $n\in\N$.

Now, we adapt to our situation the procedure in \cite[Proposition 3.5]{DapGold01} and \cite[Proposition 2.15]{LorLunZam10} (see also \cite[Proposition 2.6]{LorLunSch16Str}).
We fix $\f\in C_b(\Rd;\R^m)$, $n\in\N$, and observe that
\begin{align}
\frac{d}{dt}\|\vartheta_n\T(t)\f\|_{2,\bm\mu}^2
= & 2\sum_{j=1}^m\int_{\Rd}\vartheta_n^2(\T(t)\f)_j{\mathcal A}(\T(t)\f)_j d\mu+2\int_{\Rd}\!\vartheta_n^2\langle C\T(t)\f,\T(t)\f\rangle d\mu
\notag \\
\leq & 2\sum_{j=1}^m\int_{\Rd}\vartheta_n^2(\T(t)\f)_j{\mathcal A}((\T(t)\f)_j)d\mu
\label{progetto-2}
\end{align}
for any $t>0$. A straightforward computation reveals that
\begin{align*}
\vartheta_n{\mathcal A}(\T(t)\f)_j={\mathcal A}(\vartheta_n(\T(t)\f)_j)-(\T(t)\f)_j){\mathcal A}\vartheta_n
-2\langle Q\nabla\vartheta_n,\nabla_x(\T(t)\f)_j\rangle,
\end{align*}
which we replace in \eqref{progetto-2}. Taking \eqref{progetto-1} into account, we get
\begin{align}
\frac{d}{dt}\|\vartheta_n\T(t)\f\|_{2,\bm\mu}^2
\le & -2\sum_{j=1}^m\int_{\Rd}\vartheta_n|Q^{\frac{1}{2}}\nabla_x(\vartheta_n\T(t)\f)_j|^2 d\mu-2\int_{\Rd}\vartheta_n{\mathcal A}\vartheta_n|\T(t)\f|^2d\mu\notag\\
&-4\sum_{j=1}^m\int_{\Rd}\vartheta_n (\T(t)\f)_j\langle Q\nabla\vartheta_n,\nabla_x(\T(t)\f)_j\rangle d\mu.
\label{progetto-3}
\end{align}
Note that
\begin{align*}
&\bigg |\int_{\Rd}\vartheta_n (\T(t)\f)_j\langle Q\nabla\vartheta_n,\nabla_x(\T(t)\f)_j\rangle d\mu\bigg |\\
\le &\bigg |\int_{\Rd}(\T(t)\f)_j\langle Q\nabla\vartheta_n,\nabla_x(\vartheta_n (\T(t)\f)_j)\rangle d\mu\bigg |+\int_{\Rd}|(\T(t)\f)_j|^2|Q^{\frac{1}{2}}\nabla\vartheta_n|^2 d\mu\\
\le &\bigg (\int_{\Rd}|Q^{\frac{1}{2}}\nabla_x(\vartheta_n (\T(t)\f)_j)|^2d\mu\bigg )^{\frac{1}{2}}
\bigg (\int_{\Rd}|(\T(t)\f)_j|^2|Q^{\frac{1}{2}}\nabla\vartheta_n|^2 d\mu\bigg )^{\frac{1}{2}}\\
&+\int_{\Rd}|(\T(t)\f)_j|^2|Q^{\frac{1}{2}}\nabla\vartheta_n|^2 d\mu\\
\le &\frac{1}{4}\int_{\Rd}|Q^{\frac{1}{2}}\nabla_x(\vartheta_n (\T(t)\f)_j)|^2 d\mu+
2\int_{\Rd}((\T(t)\f)_j)^2|Q^{\frac{1}{2}}\nabla\vartheta_n|^2 d\mu
\end{align*}
so that, we can continue estimate \eqref{progetto-3} and obtain
\begin{align*}
\frac{d}{dt}\|\vartheta_n\T(t)\f\|_{2,\bm\mu}^2
\le & -\sum_{j=1}^m\int_{\Rd}\vartheta_n|Q^{\frac{1}{2}}\nabla_x(\vartheta_n\T(t)\f)_j|^2d\mu
-2\int_{\Rd}\vartheta_n\langle b,\nabla\vartheta_n\rangle |\T(t)\f|^2d\mu\notag\\
&+2\|\f\|_{\infty}^2\int_{\Rd}\big (|{\rm Tr}(QD^2\vartheta_n)|+4|Q^{\frac{1}{2}}\nabla\vartheta_n|^2\big )d\mu.
\end{align*}

Using \eqref{nonsichiama}(ii) we can estimate
\begin{align*}
-\vartheta_n(x)\langle {\bf b}(x),\nabla\vartheta_n(x)\rangle |(\T(t)\f)(x)|^2
=&\vartheta_n(x)|\vartheta'(n^{-1}\!|x|)|\langle {\bf b}(x),x\rangle \frac{1}{n|x|}|(\T(t)\f)(x)|^2\\
\le &c\vartheta_n(x)|\vartheta'(n^{-1}\!|x|)|\frac{(1\!+\!|x|^2)}{n|x|}\varphi(x)|(\T(t)\f)(x)|^2\\
%\le &c\|\vartheta'\|_{\infty}(1+4n^2)n^{-2}\|\f\|_{\infty}\varphi(x)\chi_{B_{2n}\setminus B_n}(x)\\
\le &5c\|\vartheta'\|_{\infty}\|\f\|_{\infty}^2\varphi(x)\chi_{B_{2n}\setminus B_n}(x)
\end{align*}
for any $x\in\Rd$.
Hence,
\begin{eqnarray*}
-\int_{\Rd}\vartheta_n\langle b,\nabla\vartheta_n\rangle |(\T(t)\f)|^2d\mu
\le 5c\|\vartheta'\|_{\infty}\|\f\|_{\infty}^2\int_{B_{2n}\setminus B_n}\varphi d\mu
=:a_n
\end{eqnarray*}
and the sequence $(a_n)$ vanishes as $n\to +\infty$, since the function $\varphi$ belongs to $L^1_{\mu}(\Rd)$ (see \cite[Chapther 9]{newbook}). Similarly, using
\eqref{nonsichiama}(i), we can show that
\begin{eqnarray*}
\lim_{n\to+\infty}\int_{\Rd}\big (|{\rm Tr}(QD^2\vartheta_n)|+4|Q^{\frac{1}{2}}\nabla\vartheta_n|^2\big )d\mu=0.
\end{eqnarray*}
Summing up, we have shown that
\begin{align}
\frac{d}{dt}\|\vartheta_n\T(t)\f\|_{2,\bm\mu}^2
\le & -\sum_{j=1}^m\int_{\Rd}\vartheta_n|Q^{\frac{1}{2}}\nabla_x(\vartheta_n\T(t)\f)_j|^2 d\mu+b_n
\label{progetto-4}
\end{align}
for some sequence $(b_n)$ which converges to $0$ as $n\to +\infty$.
Integrating \eqref{progetto-4} from $0$ to $t>0$, we conclude that
\begin{align*}
\int_0^tds\sum_{j=1}^m\int_{\Rd}\vartheta_n|Q^{\frac{1}{2}}(\nabla_x(\vartheta_n\T(s)\f)_j)|^2 d\mu \leq \|\f\|_{L^2_\mu}^2
+b_nt.
\end{align*}
Applying Fatou lemma to the previous formula, we deduce that
\begin{align*}
\int_0^tds\sum_{j=1}^m\int_{\Rd}|Q^{\frac{1}{2}}(\nabla_x(\T(s)\f)_j)|^2d\mu \leq \|\f\|_{L^2_\mu(\Rd;\R^m)}^2.
\end{align*}
It thus follows that the function $|Q^{1/2}\nabla_x(\T(\cdot)\f)_j|$ belongs to $L^2([0,+\infty);L^2_{\mu}(\Rd))$ for any $j=1,\ldots,m$. In particular,
if we set
\begin{align*}
h(t)=\sum_{j=1}^m\int_{\Rd}|\nabla_x(\T(t)\f)_j|^2d\mu,
\end{align*}
then $h$
belongs to $L^1((0,+\infty))$ and its $L^1$-norm is bounded by $\mu_0^{-1}\|\f\|^2_{2,\bm\mu}$, where $\mu_0$ denotes
the infimum over $\Rd$ of the minimum eigenvalue $\mu_Q(x)$ of the matrix $Q(x)$.

We now assume that $\f\in C^{3+\alpha}_c(\Rd;\R^m)$. Since the coefficients of the operator $\A$ are in $C^{1+\alpha}_{\rm loc}(\Rd)$, the function
$D_j(\T(\cdot)\f)_i$ is differentiable with respect to time (see Theorem \ref{teo-app})
 and $D_t(D_j(\T(\cdot)\f)_i)=D_j(D_t(\T(\cdot)\f)_i)
=D_j(\A\T(\cdot)\f)_i$ on $(0,+\infty)\times\Rd$. By \cite[Proposition 3.2]{DelLor11OnA}, 
$\T(t)\A\f=\A\T(t)\f$ for any $t>0$. If thus follows that $D_t(D_j(\T(\cdot)\f)_i)=D_j(\T(t)\A\f)_i$.

For any $n\in\N$, let us introduce the function $h_n:[0,+\infty)\to\R$ defined by
\begin{eqnarray*}
h_n(t)=\sum_{j=1}^m\int_{\Rd}\vartheta_n|\nabla_x(\T(t)\f)_j|^2d\mu,\qquad\;\,t>0.
\end{eqnarray*}
As it is immediately seen, $h_n$ converges to $h$ in $L^1((0,+\infty))$ and pointwise in $[0,+\infty)$.
Moreover, applying the dominated convergence and taking into account that the functions $D_t(D_j(\T(\cdot)\f)_i)$ and $D_j(\T(\cdot)\f)_i$
are continuous in $(0,+\infty)\times\Rd$, we can show that $h_n$ is differentiable in $(0,+\infty)$ and
\begin{eqnarray*}
h_n'(t)=\sum_{j=1}^m\int_{\Rd}\vartheta_n \langle \nabla_x(\T(t)\f)_j,\nabla_x(\T(t)\A\f)_j\rangle d\mu,\qquad\;\,t>0.
\end{eqnarray*}
By applying Cauchy-Schwarz inequality, we deduce that $\langle \nabla_x(\T(\cdot)\f)_j,\nabla_x(\T(\cdot)\A\f)_j\rangle$ belongs to $L^1((0,+\infty)\times\Rd;dt\times d\mu)$ for any $j=1,\ldots,m$.
Hence, the above results show that $h_n'$ converges to the function
$\sum_{j=1}^m\int_{\Rd} \langle \nabla_x(\T(\cdot)\f)_j,\nabla_x(\T(\cdot)\A\f)_j\rangle d\mu$ as $n\to +\infty$.
Since $W^{1,1}((0,+\infty))\hookrightarrow C_b([0,+\infty))$ and $h_n$ converges in $W^{1,1}((0,+\infty))$, it follows that
$h\in W^{1,1}((0,+\infty))$ and $h_n$ converges to $h$ uniformly in $(0,+\infty)$. In particular, $h$ vanishes as $t\to +\infty$.
\end{proof}

Now we study the asymptotic behaviour of $\T(t)$. To this aim, using the same notation as in the proof of Theorem \ref{exi_inv_meas},
for any $\f\in L^1_{\bm \mu}(\Rd;\R^m)$  we set
\begin{eqnarray*}
{\mathcal M}_\f= \sum_{j=1}^m\int_{\Rd}f_jd\mu_j,
\end{eqnarray*}
where $\mu_j=\xi_j\mu$ for any $j=1,\ldots,m$.

\begin{thm}\label{asy_thm}
Let Hypothesis $\ref{hyp-nonsichiama}$ be satisfied. Then, $\T(t)\f$ converges to ${\mathcal M}_{\f}\xi$ locally uniformly in $\Rd$ as $t\to +\infty$,
for any $\f\in B_b(\Rd;\Rm)$.
Further, if  $\f\in L^p_{\bm\mu}(\Rd;\Rm)$, then the function $\T(t)\f$ converges to ${\mathcal M}_{\f}\xi$
in $L^p_{\bm\mu}(\Rd;\R^m)$ as $t\to +\infty$.
\end{thm}

\begin{proof}
The last statement is a straightforward consequence of the first one.
Indeed, if $\f\in C_b(\Rd;\Rm)$, then the first statement and the dominated convergence theorem immediately show that
$\T(t)\f$ converges to ${\mathcal M}_{\f}\xi$ in $L^p_{\bm\mu}(\Rd;\R^m)$ as $t\to +\infty$.
In the general case when $\f\in L^p_{\bm\mu}(\Rd;\Rm)$, we fix a sequence $(\f_n)\subset C_b(\Rd;\Rm)$ converging to $\f$ in $L^p_{\bm\mu}(\Rd;\Rm)$ as $n\to +\infty$.
Taking \eqref{festa} into account, we can estimate
\begin{align*}
\|\T(t)\f-{\mathcal M}_{\f}\xi\|_{L^p_{\bm\mu}(\Rd;\R^m)}\le &2^{\frac{p-1}{p}}\|\f-\f_n\|_{L^p_{\bm\mu}(\Rd;\R^m)}+ \|\T(t)\f_n-{\mathcal M}_{\f_n}\xi\|_{L^p_{\bm\mu}(\Rd;\R^m)}\\
&+ |{\mathcal M}_{\f_n}-{\mathcal M}_{\f}|\bigg (\sum_{i=1}^m\xi_i^p\mu_i(\Rd)\bigg )^{\frac{1}{p}}.
\end{align*}
Letting first $t$ and then $n$ tend to $+\infty$, we easily conclude that
$\T(t)\f$ converges to ${\mathcal M}_{\f}$ in $L^p_{\bm\mu}(\Rd;\R^m)$ also in this case.

In view of the strong Feller property of the semigroup $(\T(t))$ (see Section \ref{sect:2}), it suffices to prove the first statement for functions $\f\in C_b(\Rd;\R^m)$.
Actually, we can limit ourselves to considering functions $\f\in C^{3+\alpha}_c(\Rd;\R^m)$. Indeed, each $\f\in C_b(\Rd;\R^m)$ is the local uniform limit in $\Rd$
of a sequence $(\f_n)\subset C^{3+\alpha}_c(\Rd;\R^m)$. By Proposition \ref{prop-rossana}, up to a subsequence $\T(\cdot)\f_n$ converges to $\T(\cdot)\f$ uniformly
in $[0,+\infty)\times\overline B_R$ for any $R>0$. Hence, we can estimate
\begin{align*}
\|\T(t)\f-{\mathcal M}_{\f}\xi\|_{C_b(\overline B_R;\R^m)}\le &\sup_{t\ge 0}\|\T(t)\f-\T(t)\f_n\|_{C_b(\overline B_R;\R^m)}\\
&+\|\T(t)\f_n-{\mathcal M}_{\f_n}\xi\|_{C_b(\overline B_R;\R^m)}+|{\mathcal M}_{\f_n}-{\mathcal M}_{\f}|
\end{align*}
for any $t>0$ and $R>0$ (take Remark\ \ref{rmk-102} into account). Hence, if $\T(t)\f_n$ converges to ${\mathcal M}_{\f_n}\xi$ locally uniformly in $\Rd$ for any $n\in\N$ as $t\to +\infty$,
then letting $t$ and $n$ tend to $+\infty$ in the previous estimate we conclude that
$\T(t)\f$ converges to ${\mathcal M}_{\f}\xi$ as $t\to +\infty$, locally uniformly in $\Rd$.

Fix a function $\f\in C^{3+\alpha}_c(\Rd;\R^m)$ and a sequence $(t_n)$ diverging to $+\infty$ such that $t_n>1$ for any $n\in\N$.
Since the sequence $(\T(t_n-1)\f)$ is bounded, by Proposition \ref{prop-rossana} it follows that there exists a subsequence
$(t_{n_k})$ such that $\T(t_{n_k})\f=\T(1)\T(t_{n_k}-1)\f$ converges to some function $\g\in C_b(\Rd;\R^m)$, locally uniformly on $\Rd$.
Clearly, $(\T(t_{n_k})\f)_j$ converges to $g_j$ also in $L^2_{\mu}(\Rd)$ for any $j=1,\ldots,m$. We claim that $\g$ is constant.
Indeed, from Proposition \ref{metropolitana}, it follows that $\||J_x\T(t_{n_k})\f|\|_{L^2_{\mu}(\Rd)}$ vanishes as $k\to +\infty$. Consequently,
$|J_x\g|\equiv 0$ and $\g$ is constant. Since for any $t>0$ the
function $\T(t)\g$ is the $L^2_{\bm\mu}$-limit of the sequence $(\T(t+t_{n_k})\f)$ as $k \to +\infty$,
taking again Proposition \ref{metropolitana} into account and arguing as above, we deduce that the function $\h=\T(\cdot)\g$ is independent of $x$ and
$D_t \h(t)=C(x)\h(t)$ for any $t>0$ and $x\in\Rd$. Hence, we can fix $x=0$.

Denote by $\lambda_1=0$, $\lambda_2,\ldots,\lambda_s$ ($s\le m$) the eigenvalues of the matrix $C(0)$ and by $a_i$ and $q_i$ ($i=1,\ldots,s$), respectively, their
algebraic and geometric multiplicities.
By the Jordan normal form theorem, we can write $C(0)= PJP^{-1}$ for some invertible matrix $P$ (with entries $p_{ij}$)
and some block matrix $J= \textrm{diag}(J_1,\ldots, J_s)$. Each matrix $J_i$ has dimension $a_i$ and itself splits into
$q_i$ sub-blocks $J_{ij}=\lambda_iI+N_{ij}$ for some matrix $N_{ij}$ such that $(N_{ij})_{hk}=\delta_{h+1,k}$ for each $h$ and $k$. In particular, if $n_{ij}$ denotes the dimension of the matrix $N_{ij}$
then $N_{ij}^{n_{ij}}$ is the trivial matrix.
Based on this remark, we can infer that $e^{tJ_{ij}}=e^{t\lambda_{i}}e^{tN_{ij}}$, where $e^{tN_{ij}}$ is an upper triangular matrix with
\begin{eqnarray*}
(e^{tN_{ij}})_{hk}=\frac{t^{k-h}}{(k-h)!}, \qquad\;\, k \ge h.
\end{eqnarray*}
Taking Lemma \ref{carusorompi} into account, which shows that ${\rm Re}\lambda_h<0$ for any $h=2,\ldots,s$,
it thus follows that the norm of the matrix $e^{tJ_h}$ exponentially decreases to zero as $t\to +\infty$ for any $h=2,\ldots,s$.
As a byproduct of all the above remarks, if we set $P^{-1}\g=\eta$, then for any $i=1,\ldots,m$ we can write
\begin{align*}
(\T(t)\g)_{i}= &(P e^{tJ}P^{-1}g)_{i}
=\sum_{j,h=1}^{m} p_{{i}j}(e^{tJ})_{jh}\eta_h=\sum_{j,h=1}^{a_1}p_{ij}(e^{tJ_1})_{jh}\eta_h+o(1,+\infty)
\end{align*}
for any $t>0$ where, following Landau's formalism, we have denoted by $o(1,+\infty)$ a function which vanishes as $t$ tends to $+\infty$.

Let us rewrite the sum in the last side of the previous formula, in a much more convenient way. For this purpose, we set $N_0=0$, $N_k=\sum_{j=1}^kn_{1j}$ for $k=1,\ldots,a_1$ and observe that
\begin{align*}
\sum_{j,h=1}^{a_1}p_{ij}(e^{tJ_1})_{jh}\eta_h=&\sum_{k=1}^{q_1}\sum_{j,h=N_{k-1}}^{N_k}p_{ij}(e^{tJ_{1k}})_{jh}\eta_h
=\sum_{k=1}^{q_1}\sum_{j,h=N_{k-1}+1}^{N_k}p_{ij}(e^{tN_{1k}})_{jh}\eta_h\\
=&\sum_{k=1}^{q_1}\sum_{j=N_{k-1}+1}^{N_k}\sum_{h=j}^{N_k}p_{ij}\frac{t^{h-j}}{(h-j)!}\eta_h\\
=&\sum_{j=1}^{a_1}p_{ij}\eta_j
+\sum_{k=1}^{q_1}\sum_{j=N_{k-1}+1}^{N_k}\sum_{h=j+1}^{N_k}p_{ij}\frac{t^{h-j}}{(h-j)!}\eta_h.
\end{align*}
The last term in the previous chain of inequalities is a polynomial vanishing at zero, which we denote by $q$.
Summing up, we have shown that
\begin{align}
(\T(t)\g)_i
=\sum_{j=1}^{a_1}p_{ij}\eta_j+q(t)+o(1,+\infty),\qquad\;\,t>0.
\label{form-0}
\end{align}
Since, again by Proposition \ref{prop-rossana}, up to a subsequence $\T(t_{n_k})\g$ converges locally uniformly on $\Rd$ as $k \to +\infty$ , from \eqref{form-0} we deduce that
$q$ is the null polynomial, so that
\begin{align*}
(\T(t)\g)_i= \sum_{j=1}^{a_1}p_{ij}\eta_j+o(1,+\infty)
\end{align*}
for $t>0$. The above formula shows that $\T(t)\g$ converges as $t\to +\infty$.
By the proof of Theorem \ref{exi_inv_meas}, we know that the average of $\T(\cdot)\g$ over the interval $[0,t]$ converges as $t\to +\infty$ to ${\mathcal M}_{\g}\xi$. Since the function
$\T(\cdot)\g$ is bounded and converges at infinity, we conclude that $\T(t)\g$ converges to ${\mathcal M}_{\g}\xi$ as $t\to +\infty$.
Actually ${\mathcal M}_{\g}={\mathcal M}_{\f}$. Indeed, by invariance property of the measures $\mu_i$ we can write
\begin{eqnarray*}
\sum_{j=1}^m\int_{\Rd}(\T(t_{n_k})\f)_jd\mu_j=\sum_{j=1}^m\int_{\Rd}f_jd\mu_j
\end{eqnarray*}
and letting $k$ tend to $+\infty$ by dominated convergence, we obtain
\begin{eqnarray*}
\sum_{j=1}^m\int_{\Rd}g_jd\mu_j=\sum_{j=1}^m\int_{\Rd}f_jd\mu_j
\end{eqnarray*}
i.e., ${\mathcal M}_{\g}={\mathcal M}_{\f}$.

Now, we can prove that $\T(\cdot)\f$ converges to ${\mathcal M}_{\f}\xi$ locally uniformly in $\Rd$ as $t\to +\infty$.
For this purpose, we fix $R>0$ and estimate
\begin{align*}
&\|\T(t)\f-{\mathcal M}_{\f}\xi\|_{C(\overline B_R)}\\
\le &\|\T(t-t_{n_k})(\T(t_{n_k})\f-\g)\|_{C(\overline B_R)}
+\|\T(t-t_{n_k})\g-{\mathcal M}_{\f}\xi\|_{C(\overline B_R)}\\
\le &\sup_{s>0}\|\T(s)(\T(t_{n_k})\f-\g)\|_{C(\overline B_R)}
+\|\T(t-t_{n_k})\g-{\mathcal M}_{\f}\xi\|_{C(\overline B_R)}
\end{align*}
for any $t$ and $k\in\N$ such that $t-t_{n_k}>0$. Fix $k\in\N$. Letting $t\to +\infty$ in the above estimate gives
\begin{eqnarray*}
\limsup_{t\to +\infty}\|\T(t)\f-{\mathcal M}_{\f}\xi\|_{C(\overline B_R)}
\le \sup_{s>0}\|\T(s)(\T(t_{n_k})\f-\g)\|_{C(\overline B_R)}
\end{eqnarray*}
for any $k\in\N$. Finally, using Proposition \ref{prop-rossana}, we can let $k$ tend to $+\infty$ and conclude that
$\limsup_{t\to +\infty}\|\T(t)\f-{\mathcal M}_{\f}\xi\|_{C(\overline B_R)}=0$, and we are done.
\end{proof}

\begin{example}{\rm
Let $\A$ be as in Example \ref{example-1}. It is easy to show that, for any $\sigma>0$, the function $\varphi_{\sigma}:\Rd\to\R$ ($\sigma>0$), defined by $\varphi_{\sigma}(x)=(1+|x|^2)^{\sigma}$, for any $x\in \Rd$, satisfies Hypothesis \ref{hyp-base}(iv) too. Hence, if $\sigma>(\gamma-1)^+$ then also
Hypothesis \ref{hyp-nonsichiama} is satisfied and Theorem \ref{asy_thm} can be applied.}
\end{example}

\appendix
\section{Regularity of solutions to parabolic problems}

\begin{thm}
\label{teo-app}
Let $\Omega$ be a domain of $\Rd$ and let $\A={\rm Tr}(QD^2)+\langle {\bf b},\nabla\rangle+C$ with the entries of the matrix-valued functions $C$, $Q$ and those of the vector-valued function ${\bf b}$ in $C^{k+\alpha}_{\rm loc}(\Omega)$ for some $k\in\N$ and $\alpha\in (0,1)$.
Let $\uu\in C^{1+\alpha/2,2+\alpha}_{\rm loc}([0,T]\times\Omega;\Rm)$ be such that $\uu(0,\cdot)\in C^{2+k+\alpha}_{\rm loc}(\Rd;\Rm)$, $D_t\uu-\A\uu\in C^{\alpha/2,k+\alpha}_{\rm loc}([0,T]\times\Omega;\Rm)$. Then,
$\uu\in C^{1+\alpha/2,2+k+\alpha}_{\rm loc}([0,T]\times\Omega;\Rm)$, $D_t\uu\in C^{\alpha,k+\alpha}_{\rm loc}([0,T]\times\Omega;\Rm)$ and
$D^{\beta}_xD_t\uu=D_tD^{\beta}_x\uu$ in $(0,T)\times\Omega$ for any $|\beta|\le k$.
\end{thm}

\begin{proof}
We begin by the case $k=1$, which is the core of the proof. We fix two connected open sets $\Omega_1$ and $\Omega_2$ with $\Omega_1\Subset\Omega_2\Subset\Omega$, and a function $\vartheta\in C^\infty(\Rd)$
such that $\chi_{\Omega_1}\le\vartheta\le\chi_{\Omega_2}$. Next, we set $\g=D_t\uu-\A\uu$, denote by $\vv$ the trivial extension of the function $\vartheta \uu$ to the whole $[0,T]\times\Rd$ and, for any $j=1,\ldots,d$ and $h\in\R\setminus\{0\}$, introduce the operator $\Delta_{h,j}$ defined on smooth functions $\psi$ by $\Delta_{h,j}\psi=h^{-1}(\psi(\cdot+he_j)-\psi)$,
which is bounded from $C^{1+\alpha}_b(\Rd)$ into $C^{\alpha}_b(\Rd)$. The function
$\vv_{h,j}=\Delta_{h,j}\vv:=(\Delta_{h,j}v_1,\ldots,\Delta_{h,j}v_m)$ belongs to $C^{1+\alpha/2,2+\alpha}_{\rm loc}([0,T)\times\Rd;\Rm)$ and
$D_t\vv_{h,j}=\hat\A\vv_{h,j}+\g_{h,j}$ in $(0,T)\times\Rd$,
where $\hat\A$ is defined as the operator $\A$, with $Q$, ${\bf b}$ and $C$ being replaced by
$\hat Q=\eta Q+(1-\eta)I$, $\hat b=\eta b$, $\hat C=\eta C$ for some function $\eta\in C^{\infty}_c(\Omega)$ such that $\eta\equiv 1$ in $\Omega_2$,
and $\g_{h,j}=(\g_{h,j}^{(1)},\ldots,\g_{h,j}^{(m)})$, where
\begin{align*}
\g_{h,j}^{(k)}=&\Delta_{h,j}[\vartheta g_k-u_k{\mathcal A}_0\vartheta-2\langle\hat Q\nabla\eta,\nabla_xu_k\rangle]+
{\rm Tr}((\Delta_{h,j}\hat Q)D^2_xv_k(\cdot,\cdot+he_j))\\
&+\langle \Delta_{h,j}\hat {\bf b},\nabla v_k(\cdot,\cdot+he_j)\rangle
+((\Delta_{h,j}\hat C)\vv(\cdot,\cdot+he_j))_k,
\end{align*}
for any $k=1,\ldots,m$, with ${\mathcal A}_0={\rm Tr}(\hat QD^2)+\langle\hat{\bf b},\nabla\rangle$.
Being a bounded perturbation of a diagonal operator, which generates an analytic semigroup in $C_b(\Rd;\Rm)$,
the operator $\hat\A$ itself is the generator of an analytic semigroup, which has maximal $C^{\alpha}$-regularity.
Since, due to our assumptions, $\vv_{h,j}(0,\cdot)$ and $\g_{h,j}$ belong to $C^{2+\alpha}_b(\Rd;\Rm)$ and
$C^{\alpha/2,\alpha}_b([0,T]\times\Rd;\Rm)$, respectively, this means that we can estimate
\begin{align*}
\|\vv_{h,j}\|_{C^{1+\alpha/2,2+\alpha}_b([0,T]\times\Rd;\Rm)}\le & c_1(\|\vv_{h,j}(0,\cdot)\|_{C^{2+\alpha}_b(\Rd;\Rm)}\!+\!\|\g_{h,j}\|_{C^{\alpha/2,\alpha}_b([0,T]\times\Rd;\Rm)})\\
\le & c_2(\|\uu(0,\cdot)\|_{C^{3+\alpha}_b(\Omega_3;\Rm)}+\|\g\|_{C^{\alpha/2,1+\alpha}_b([0,T]\times\Omega_3;\Rm)}\\
&\phantom{c_2(\,}+\|\uu\|_{C^{\alpha/2,2+\alpha}_b([0,T]\times\Omega_3;\Rm)})
\end{align*}
for any $|h|<h_0:={\rm dist}(\Omega,\partial\Omega_2)$ and some positive constants $c_1$ and $c_2$, independent of the functions involved, where
$\Omega_3=\Omega_2+B_{h_0}$.
Taking into account that $\vv_{h,j}$ converges to $D_j\vv$ pointwise on $[0,T]\times\Rd$ as $h\to 0$, a compactness argument shows that $D_j\vv$ belongs to $C^{1+\alpha/2,2+\alpha}_b([0,T]\times\Rd;\Rm)$. As a byproduct, $D_j\uu\in C^{1+\alpha/2,2+\alpha}_{\rm loc}([0,T]\times\Omega;\Rm)$. We also deduce that $D_t\uu$ is continuously differentiable in $[0,T]\times\Omega$ with respect to the variable $x_j$. This is enough to infer that
$D_tD_j\uu=D_jD_t\uu$ in $[0,T]\times\Omega$.

Now, suppose that the claim holds for some $k>1$ and set $\g=D_t\uu-\A\uu$. Differentiating the equation $D_t\uu=\A\uu+\g$ $k$-times with respect to the spatial variables, we conclude that, for any $\beta$, with length $k$, the function $\ww=D^{\beta}_x\uu$ solves the differential equation $D_t\ww=\A\ww+\g_{\beta}$, where $\g_{\beta}$ is a linear combination of the spatial derivatives of
$\uu$ up to the order $k+1$ with coefficients which are the derivatives of the coefficients up to the order $k$ of operator $\A$. As a consequence, $\g_{\beta}$ belongs to $C^{\alpha/2,1+\alpha}_{\rm loc}([0,T]\times\Omega;\Rm)$ and from the first part of the proof we conclude that $D^{\beta}_x\uu\in C^{1+\alpha/2,2+\alpha}_{\rm loc}([0,T]\times\Omega;\Rm)$. The arbitrariness of $\beta$ implies that $\uu\in C^{1+\alpha/2,3+k+\alpha}_{\rm loc}([0,T]\times\Omega;\Rm)$.
\end{proof}


\begin{thebibliography}{99}

\bibitem{acqui}
{P. Acquistapace},
{\it Evolution operators and strong solutions of abtract parabolic equations}, Differential Integral Equations {\bf 1} (1988), 433-457.

\bibitem{AALT}
D. Addona, L. Angiuli, L. Lorenzi, G. Tessitore,
{\it On coupled systems of Kolmogorov equations with applications to stochastic differential games.}
ESAIM: Control. Optim. Calc. Var. (to appear), doi: 10.1051/cocv/2016019.

\bibitem{AFP}
L. Ambrosio, N. Fusco, D. Pallara,
\newblock{Functions of bounded variation and free discontinuity problems.}
\newblock{ Oxford University Press, USA, 2000.}

\bibitem{AngLor10Com}
L. Angiuli, L. Lorenzi,
\newblock{\em Compactness and invariance properties of evolution operators
associated to Kolmogorov operators with unbounded coefficients},
\newblock{ J. Math. Anal. Appl.} {\bf 379} (2011), 125-149.


\bibitem{AngLorPal}
L. Angiuli, L. Lorenzi, D. Pallara,
\newblock{\em $L^p$ estimates for parabolic systems with unbounded coefficients coupled at zero and first order},
\newblock{ J. Math. Anal. Appl.} {\bf 444} (2016), 110-135.

\bibitem{BerFor04Gra}
M. Bertoldi, S. Fornaro,
\newblock{\em Gradient estimates in parabolic problems with unbounded coefficients},
\newblock{Studia Math.} {\bf 165} (2004), 221-254.

\bibitem{BGT}
{\rm V. Betz, B.D. Goddard, S. Teufel},
\newblock{\em Superadiabatic transitions in quantum molecular dynamics},
Proc. R. Soc. A \textbf{465} (2009), 3553-3580.


\bibitem{DapGold01}
G. Da Prato, B. Goldys,
\newblock{\it Elliptic operators on $\Rd$ with unbounded coefficients},
J. Differential Equations {\bf 172} (2001), 333-358.

%\bibitem{daprato-zab-ergodicity}
%{\rm G. Da Prato, J. Zabczyk,
%{Ergodicity for infinite dimensional systems, London Mathematical
%Society, Cambridge University Press, 1996.}}

\bibitem{Dall}
G. Dall'Ara,
\newblock{\em Discreteness of the spectrum of Schr\"odinger operators with non-negative matrix-valued potentials},
J. Funct. Anal. \textbf{268} (2015), 3649-3679.

\bibitem{DelLor11OnA}
S. Delmonte, L. Lorenzi,
{\em On a class of weakly coupled systems of elliptic operators with unbounded coefficients}, Milan J. Math. {\bf 79} (2011), 689-727.

\bibitem{EN}
K-J. Engel, R. Nagel,
One-parameter semigroups for linear evolution equations, Graduate Texts in Mathematics \textbf{194},
Springer-Verlag, New York, 2000.

\bibitem{Ha-Rh}
T. Hansel, A. Rhandi,
{\em The Oseen-Navier-Stokes flow in the exterior of a rotating obstacle: {T}he non-autonomous case},
J. Reine Angew. Math. \textbf{694} (2014), 1-26.

\bibitem{Ha-He}
F. Haslinger, B. Helffer,
\newblock{\em Compactness of the solution operator to $\overline{\partial}$ in weighted $L^2$-spaces},
J. Funct. Anal. \textbf{243} (2007), 679-697.

\bibitem{hetal09}
M. Hieber, L. Lorenzi, J. Pr{\"u}ss, A. Rhandi, R. Schnaubelt,
\emph{Global properties of generalized Ornstein-Uhlenbeck operators on $L^p(\mathbb{R}^N,\mathbb{R}^N)$ with more than linearly growing coefficients},
J. Math. Anal. Appl. \textbf{350} (2009), 100-121.

\bibitem{HRS}
M. Hieber, A. Rhandi, O. Sawada,
\newblock{\em The Navier-Stokes flow for globally {L}ipschitz continuous initial data},
Res. Inst. Math. Sci. (RIMS) (2007), 159--165.

\bibitem{hieber}
M. Hieber, O. Sawada,
\newblock{\emph The {N}avier-{S}tokes equations in {$\mathbb{R}^n$} with linearly growing initial data},
Arch. Ration. Mech. Anal. \textbf{175} (2005), 269-285.

\bibitem{KunLorLun09Non}
M. Kunze, L. Lorenzi, A. Lunardi,
{\it Nonautonomous Kolmogorov parabolic equations
with unbounded coefficients}, Trans. Amer. Math. Soc. {\bf 362} (2010), 169-198.

\bibitem{LadSolUra68Lin}
O.A. Lady\v{z}henskaja, V.A. Solonnikov, N.N. Ural'ceva,
\newblock{Linear and quasilinear equations of parabolic type}, Nauka, Moscow, 1967.
\newblock{English transl.: American Mathematical Society}, Providence, R.I. 1968.

\bibitem{newbook}
L. Lorenzi,
\newblock{Analytical methods for Kolmogorov equations. Second Edition},
\newblock{Chapman Hall/CRC Press, 2016}.

\bibitem{LorLunSch16Str}
L. Lorenzi, A. Lunardi, R. Schnaubelt,
\newblock{\em Strong convergence of solutions to nonautonomous Kolmogorov equations},
Proc. Amer. Math. Soc. {\bf 144} (2016), 3903-3917.

\bibitem{LorLunZam10}
L. Lorenzi, A. Lunardi, A. Zamboni,
\newblock{\it Asymptotic behavior in time periodic parabolic problems with unbounded coefficients},
J. Differential equations {\bf 249} (2010), 3377-3418.

%\bibitem{attractors-hale}
%{\rm J. K. Hale,
%Asymptotic behavior of dissipative systems.
%Mathematical surveys and monographs, American Mathematival Society, Providence, RI (1988).}

\bibitem{Ots88Ont}
K. Otsuka,
{\it On the positivity of the fundamental solutions for parabolic systems},
J. Math. Kyoto Univ., {\bf 28} (1988), pp. 119-132.

\bibitem{gersh}
R.S. Varga,
\newblock
Matrix iterative analysis. Springer Series in Computational Mathematics \textbf{27}, Springer-Verlag, Berlin, 2000.
\end{thebibliography}
\end{document}